\documentclass[12pt]{amsart}
\usepackage{graphicx}
\usepackage{esint}
\usepackage{amssymb, mathrsfs, url, amsfonts, amsthm, amsmath}
\usepackage{enumerate}
\usepackage[dvipsnames]{xcolor}
\usepackage{etoolbox}
\apptocmd{\sloppy}{\hbadness 10000\relax}{}{}
\apptocmd{\sloppy}{\vbadness 10000\relax}{}{}
\usepackage{hyperref}
\usepackage[letterpaper,margin=1.1in]{geometry}

\renewcommand{\baselinestretch}{1.1}

\newtheorem{theorem}{Theorem}[section]
\newtheorem{lemma}[theorem]{Lemma}
\newtheorem{corollary}[theorem]{Corollary}
\newtheorem{proposition}[theorem]{Proposition}

\theoremstyle{definition}
\newtheorem{definition}[theorem]{Definition}

\theoremstyle{remark}
\newtheorem{remark}[theorem]{Remark}
\newtheorem{question}[theorem]{Question}

\let\inf\relax \DeclareMathOperator*\inf{\vphantom{p}inf}

\newcommand{\Haus}{\mathcal{H}}
\newcommand{\Net}{\mathcal{M}}
\newcommand{\XX}{\mathbb{X}}
\newcommand{\RR}{\mathbb{R}}
\newcommand{\ZZ}{\mathbb{Z}}

\def\res{\hbox{ {\vrule height .22cm}{\leaders\hrule\hskip.2cm} } }

\newcommand{\Child}{\mathop\mathsf{Child}\nolimits}
\newcommand{\gap}{\mathop\mathrm{gap}\nolimits}
\newcommand{\side}{\mathop\mathrm{side}\nolimits}
\newcommand{\vol}{\mathop\mathrm{vol}\nolimits}
\newcommand{\dist}{\mathop\mathrm{dist}\nolimits}
\newcommand{\interior}{\mathop\mathrm{int}\nolimits}
\newcommand{\diam}{\mathop\mathrm{diam}\nolimits}

\numberwithin{equation}{section}
\numberwithin{figure}{section}

\begin{document}

\title{Hausdorff dimension of caloric measure}
\date{July 12, 2023}
\author{Matthew Badger}
\author{Alyssa Genschaw}
\thanks{M.~Badger was partially supported by NSF grant DMS 1650546.}
\address{Department of Mathematics\\ University of Connecticut\\ Storrs, CT 06269-3009}
\email{matthew.badger@uconn.edu}
\address{Mathematics Department\\ Milwaukee School of Engineering\\ Milwaukee, WI 53202}
\email{genschaw@msoe.edu}

\begin{abstract}We examine caloric measures $\omega$ on general domains in $\RR^{n+1}=\RR^n\times\RR$ (space $\times$ time) from the perspective of geometric measure theory. On one hand, we give a direct proof of a consequence of a theorem of Taylor and Watson (1985) that the lower parabolic Hausdorff dimension of $\omega$ is at least $n$ and $\omega\ll\Haus^n$. On the other hand, we prove that the upper parabolic Hausdorff dimension of $\omega$ is at most $n+2-\beta_n$, where $\beta_n>0$ depends only on $n$. Analogous bounds for harmonic measures were first shown by Nevanlinna (1934) and Bourgain (1987). Heuristically, we show that the \emph{density} of obstacles in a cube needed to make it unlikely that a Brownian motion started outside of the cube exits a domain near the center of the cube must be chosen according to the ambient dimension.

In the course of the proof, we give a caloric measure analogue of Bourgain's alternative: for any constants $0<\epsilon\ll_n \delta<1/2$ and closed set $E\subset\RR^{n+1}$, either (i) $E\cap Q$ has relatively large caloric measure in $Q\setminus E$ for every pole in $F$ or (ii) $E\cap Q_*$ has relatively small $\rho$-dimensional parabolic Hausdorff content for every $n<\rho\leq n+2$, where $Q$ is a cube, $F$ is a subcube of $Q$ aligned at the center of the top time-face, and $Q_*$ is a subcube of $Q$ that is close to, but separated backwards-in-time from $F$: $$Q=(-1/2,1/2)^n\times (-1,0),\quad  F=[-1/2+\delta,1/2-\delta]^n\times[-\epsilon^2,0),$$ $$\text{and}\quad Q_*=[-1/2+\delta,1/2-\delta]^n\times[-3\epsilon^2,-2\epsilon^2].$$ Further, we supply a version of the strong Markov property for caloric measures.\end{abstract}

\maketitle

\setcounter{tocdepth}{1}
\tableofcontents

\section{Introduction}

The caloric measures $\{\omega^{X,t}_\Omega\}_{(X,t)\in\Omega}$ are a family of Borel probability measures on $\RR^{n}\times\RR$, with support contained in $\partial\Omega$, which yield an integral representation of solutions to the Dirichlet problem for the heat equation, $$u_{X_1X_1}+\cdots+u_{X_nX_n}=u_t.$$ For any connected, open set $\Omega\subset\RR^n\times\RR$ (space $\times$ time) and suitable boundary data $f$, the function $$u(X,t)=\int f\,d\omega^{X,t}_\Omega$$ is a classical solution of the heat equation in the interior of $\Omega$ that agrees with $f$ along an appropriate subset of the boundary of $\Omega$ at each point of continuity of $f$. Probabilistically, $\omega^{X,t}(E)$ is the probability that the trace $\{(W_{X}(s),t-s):s\geq 0\}$ of a Brownian motion $W_X$ starting at $X$, \emph{sent into the past}, exits $\Omega$ through $E\subset\partial\Omega$. One should not only imagine smooth domains such as a time-cylinder over a ball, but also rough domains with corners, cusps, fractal, or space-filling boundaries. In the most general situation, the underlying open sets $\Omega_t=\{X:(X,t)\in\Omega\}\subset\RR^n$ in the space coordinates may vary in time. The precise definition of caloric measure and theorems about solution of the Dirichlet problem are the subject of heat potential theory; see \S\ref{sec:heat} for a brief overview and the monograph \cite{Watson} by Watson for a comprehensive introduction. There is a parallel theory of harmonic measures associated to the Dirichlet problem for Laplace's equation $u_{X_1X_1}+\dots+u_{X_nX_n}=0$ on domains in $\RR^n$, which is better developed and motivates questions about caloric measure (see the book review \cite{bishop-review} for a neat introduction and the reviewed book \cite{GM} for a longer one). However, the asymmetric nature of time makes the theory of temperatures more subtle. Still one expects that there exist deep connections between the geometry of a domain and its boundary and properties of the caloric measure. For a sample of results in this direction, see \cite{Kaufman-Wu-82}, \cite{Lewis-Murray}, \cite{Sweezy-ac}, \cite{HLN}, \cite{Engelstein-parabolic}, \cite{Genschaw-Hofmann}, \cite{MP-caloric}, and the references therein.

In this paper, we examine relationships between caloric measures on arbitrary domains and parabolic Hausdorff measures, an idea which stems from \cite{TW85}. More specifically, we provide estimates on the size of null sets and carrying sets for caloric measures, which are best described using the language of geometric measure theory. Very quickly, the \emph{lower Hausdorff dimension} of a measure $\mu$ on a metric space is the infimum of Hausdorff dimensions of sets with positive measure. The \emph{upper Hausdorff dimension} of $\mu$ is the infimum of Hausdorff dimensions of sets with full measure, i.e.~sets whose complement have zero measure. We review the definitions of parabolic Hausdorff measures and dimensions and related background in more detail in \S\ref{sec:gmt}. As usual, for two measures $\mu$ and $\nu$ on a measurable space $(\XX,\mathcal{M})$, we write $\mu\ll\nu$, if $E\in\mathcal{M}$ and $\nu(E)=0$ implies $\mu(E)=0$, and write $\mu\perp\nu$, if there exists $E\in\mathcal{M}$ with $\mu(\XX\setminus E)=0$ and $\nu(E)=0$. Here is a statement of the main theorem.

\begin{theorem}[dimension of caloric measure] \label{main} Let $n\geq 1$, let $\Omega\subset\RR^{n+1}$ be a domain, where we equip $\RR^{n+1}=\RR^n\times\RR$ with the parabolic distance \begin{equation}\label{e:p-d}\dist((X,t),(Y,s))=\max\left\{\left(\textstyle\sum_1^n|X_i-Y_i|^2\right)^{1/2},|t-s|^{1/2}\right\},\end{equation} let $(X,t)\in\Omega$, and let $\omega_\Omega^{X,t}$ denote the caloric measure of $\Omega$ with pole at $(X,t)$. Then: \begin{enumerate} \item The lower Hausdorff dimension of $\omega_\Omega^{X,t}$ is at least $n$. In fact, $\omega_\Omega^{X,t}\ll\Haus^n$, where $\Haus^\rho$ denotes the $\rho$-dimensional Hausdorff measure on $\RR^{n+1}$ with the parabolic distance.
\item The upper Hausdorff dimension of $\omega_\Omega^{X,t}$ is at most $n+2-\beta_n$, where $\beta_n>0$ depends only on $n$. That is, $\omega_\Omega^{X,t}\perp \Haus^{\rho}$ for all $\rho>n+2-\beta_n$.  \end{enumerate} For (ii), we assume that nearly every point in the essential boundary $\partial_e\Omega$ is regular for the Dirichlet problem for the heat equation (see \S\ref{sec:heat}).
\end{theorem}

\begin{remark}[{interpretation of Theorem \ref{main}}] Although the topological dimension is $n+1$, the Hausdorff dimension of $\RR^{n+1}=\RR^n\times\RR$ with respect to the parabolic distance is $n+2$. In particular, while nonempty relatively open sets $U\times\{t\}$ in the space coordinates have Hausdorff dimension $n$, lines parallel to the time axis have Hausdorff dimension 2. Theorem \ref{main}(i) gives a simple geometric test for a set to have zero caloric measure; namely, every set of parabolic Hausdorff dimension less than $n$ is a caloric null set. By contrast, Theorem \ref{main}(ii) implies that any boundary with sufficiently large parabolic Hausdorff dimension must contain a subset of strictly smaller Hausdorff dimension and full caloric measure. In effect, on domains with geometrically very large boundary, the interior of the domain is insulated from heat on a substantial portion of boundary. On the whole, the theorem indicates that the study of fine properties of caloric measure should focus on subsets of the boundary of parabolic Hausdorff dimension between $n$ and $n+2-\beta_n$. Unfortunately, we do not currently know the optimal value of $\beta_n$ for any $n$, but see \S\ref{sec:constants}.\end{remark}

A version of Theorem \ref{main}(i) for harmonic measure on planar domains first appears in the paper \cite{Nevanlinna}, where harmonic measure gets its name. In particular, using logarithmic capacity, Nevanlinna proves that for domains $\Omega\subset\RR^2$, harmonic measure $\omega^X_\Omega$ is absolutely continuous with respect to the Hausdorff measure $\Haus^\phi$ on $\RR^2$ with a logarithmic gauge $\phi$. For domains $\Omega\subset\RR^n$ in dimensions $n\geq 3$, the same method using Newtonian capacity tells us that harmonic measure $\omega^X_\Omega$ is absolutely continuous with respect to $\Haus^{n-2}$ on $\RR^n$. (Moreover, by elementary properties of the Riesz capacities, the harmonic measure $\omega^X_\Omega$ vanishes on any set of $\sigma$-finite $\Haus^{n-2}$ measure.) The appendix of \cite{Badger-thesis} records an alternative proof of this fact that bypasses capacity and provides the estimate $$\omega^X_\Omega(B(x,r)) \leq \frac{r^{n-2}}{\dist(X,\partial\Omega)^{n-2}}\quad\text{for all $x\in\partial\Omega$ and $r>0$.}$$ We follow this approach and give a proof of Theorem \ref{main}(i), with an analogous estimate \eqref{e:cylinder-3} for caloric measures $\omega^{X,t}_\Omega$ on parabolic balls, in \S\ref{sec:lower-bound}. Alternatively, Theorem \ref{main}(i) is a consequence of a deeper fact by Taylor and Watson \cite{TW85} that sets of zero $n$-dimensional parabolic Hausdorff measure $\Haus^n$ on $\RR^{n+1}$ have thermal capacity zero. All of the proofs described above ultimately rely on elementary comparisons with the fundamental solutions to the Laplace and heat equations. Theorem \ref{main}(i) is sharp, because the bottom boundary of any cylindrical domain $\Omega=\Omega_0\times(0,\infty)\subset\RR^n\times\RR$ has Hausdorff dimension $n$ and positive caloric measure.

Theorem \ref{main}(ii) is inspired by a corresponding result for harmonic measure due to Bourgain. For any $n\geq 2$, define \emph{Bourgain's constant} $b_n\in[0,1]$ to be the largest number such that the upper Hausdorff dimension of harmonic measure is at most $n-b_n$ for every domain $\Omega\subset\RR^n$. We remark that to bound $b_n$ from above, one need only estimate the dimension of harmonic measure from below on a particular domain. By contrast, in order to bound $b_n$ from below, one must estimate the dimension of harmonic measure from above \emph{uniformly} on arbitrary domains. In \cite{Bourgain}, Bourgain did this and showed that $b_n>0$ for all $n\geq 3$. It is now known that $b_2=1$ by a deep theorem of Jones and Wolff \cite{Jones-Wolff} (also see Makarov \cite{Makarov} and Wolff \cite{Wolff-sigma}) and $b_n<1$ for all $n\geq 3$ by an incredible example of Wolff \cite{Wolff} (also see \cite{LVV}). However, the exact value of Bourgain's constant $b_n$ has not been determined. Bishop \cite{Bishop-questions} has conjectured that $b_n=1/(n-1)$, but there has been only a little progress towards computing $b_n$ to date. A numerical experiment by Grebenkov \emph{et.~al} \cite{grebenkov-numerical-experiment} suggests $b_3\leq 0.995$ and the authors \cite{BG-decimals} recently proved that $b_3\geq 10^{-15}$. Additional results about the Hausdorff dimension of harmonic measure on special classes of domains can be found in \cite{wu-singularity,cantor-dim,kenigpreisstoro,Badger-nullsets,amt-example,azzam-corkscrews,azzam-dimension-drop}.

To prove Theorem \ref{main}(ii), in \S\ref{sec:upper-bound}, we follow the outline of Bourgain's original proof, which utilizes fundamental tools such as the maximum principle for harmonic functions and the strong Markov property for harmonic measure together with stopping time arguments from harmonic analysis. Each of these have natural analogues in the caloric setting, which we review in \S\ref{sec:heat}. In particular, we supply a version of the strong Markov property for caloric measure that incorporates the essential and singular boundary classification (see Theorem \ref{t:strong-markov}). We expect this tool to be useful for future applications. Of further independent interest, we supply a self-contained statement and proof (see Theorem \ref{dim-lemma}) of Bourgain's strategy for estimating the dimension of a Radon measure, which appears at the end of his demonstration that $b_3>0$. Another crucial ingredient in \cite{Bourgain} is an estimate, now commonly called \emph{Bourgain's alternative}, which implies that locally any closed set in $\RR^n$ always has relatively large harmonic measure or relatively small $\rho$-dimensional Hausdorff content $\Haus^\rho_\infty$ for all $\rho>n-2$. We determine and prove the correct analogue of Bourgain's alternative with caloric measure and parabolic Hausdorff content. Because of the form of the fundamental solution to the heat equation, a short dimension-dependent time skip between the pole of caloric measure and the part of the closed set being examined is required (see Lemma \ref{alternative}).

In \S\ref{sec:constants}, we examine the connection between Bourgain's constant $b_n$ for harmonic measure and the optimal constant $\beta_n$ in Theorem \ref{main}(ii). In particular, considering the behavior of caloric measure on cylindrical domains, we prove that $\beta_n\leq b_n$ for all $n\geq 1$.

As a final note, we wish to draw attention to some related work for other operators. Akman, Lewis, and Vogel \cite{ALV15} proved that an analogue of Bourgain's theorem is  true for $p$-harmonic measures associated to the \emph{$p$-Laplacian}, a degenerate elliptic operator on $\RR^n$, when $p\geq n$. In fact, they demonstrate that $p$-harmonic measure is carried by a set of $\sigma$-finite $\Haus^{n-1}$ measure whenever $p>n$. We expect that there should be analogues of Bourgain's theorem and Theorem \ref{main}(ii) for elliptic measures and parabolic measures associated to uniformly elliptic and parabolic divergence form operators, but leave this to future investigations.
On the other hand, Sweezy \cite{Sweezy-planar,Sweezy-space} has already proved by example that there exist elliptic measures on $\RR^n$ with Hausdorff dimension arbitrarily close to $n$ and parabolic measures in $\RR^{n+1}$ with Hausdorff dimension arbitrarily close to $n+2$. For adjacent work on absolute continuity of ``harmonic measure in higher codimensions", see the paper  \cite{DEM-magic} by David, Engelstein, and Mayboroda.

\emph{Acknowledgements.}  The first author would like to thank Max Engelstein for discussions about caloric measure and Raanan Schul for suggesting Question \ref{q:sharp}.

\section{Preliminaries I: geometric measure theory} \label{sec:gmt}

\subsection{Metric spaces} For general background on Hausdorff measures, see the books of Mattila \cite{Mattila} and Rogers \cite{Rogers}. For discussion of various definitions of the dimension of a measure, we refer the reader to \cite{measure-dim}. On a metric space $\XX$, we let $\diam E$ denote the \emph{diameter} of $E\subset \XX$. We let $B(x,r)$ and $U(x,r)$ denote the \emph{closed} and \emph{open} balls with center $x\in\XX$ and radius $r>0$.

\begin{definition} Let $\XX$ be a metric space and let $\rho\in[0,\infty)$. For all $\delta\in(0,\infty]$, we define the set functions $\Haus^\rho_\delta:\mathcal{P}(\XX)\rightarrow[0,\infty]$ by $$\Haus^\rho_\delta(E)=\inf\left\{\sum_1^\infty (\diam E_i)^\rho:E\subseteq\bigcup_1^\infty E_i, \diam E_i\leq \delta\right\}\quad\text{for all }E\subset\XX;$$ we call $\Haus^\rho_\infty$ the \emph{$\rho$-dimensional Hausdorff content} on $\XX$. We define the \emph{$\rho$-dimensional Hausdorff measure} $\Haus^\rho:\mathcal{P}(\XX)\rightarrow[0,\infty]$ on $\XX$ by $\Haus^\rho(E)=\lim_{\delta\downarrow 0} \Haus^\rho_\delta(E)$ for all $E\subset\XX$.
\end{definition}

\begin{proposition}[basic facts] Let $\XX$ be a metric space. \begin{itemize}
\item For all $\rho\in[0,\infty)$, the set functions $\Haus^\rho_\delta$ are outer measures on $\XX$ and the Hausdorff measure $\Haus^\rho$ is a Borel regular outer measure on $\XX$.
\item For all $\rho\in[0,\infty)$ and $E\subset\XX$, $\Haus^\rho_\infty(E)\leq (\diam E)^\rho$.
\item For all $\rho\in[0,\infty)$ and $E\subset\XX$, $\Haus^\rho(E)=0$ if and only if $\Haus^\rho_\infty(E)=0$.
\item For all $0\leq \rho_1<\rho_2<\infty$ and $E\subset\XX$, if $\Haus^{\rho_1}(E)<\infty$, then $\Haus^{\rho_2}(E)=0$.\end{itemize}\end{proposition}

\begin{remark}In the definition of $\Haus^\rho_\delta$, the coverings of $E$ are by \emph{arbitrary subsets} of $\XX$ of diameter at most $\delta$. The outer measures $\Haus^\rho_\delta$ are typically not additive on Borel sets.\end{remark}

\begin{definition} Let $\XX$ be a metric space and let $E\subset\XX$. The \emph{Hausdorff dimension} of $E$, denoted $\dim_H E$, is the unique number in $[0,\infty]$ determined as follows. \begin{itemize}
\item If $\Haus^\rho(E)=0$ for all $\rho>0$, then $\dim_H E=0$.
\item If there exists $\sigma\in(0,\infty)$ such that $\Haus^{\rho_1}(E)=\infty$ for all $\rho_1<\sigma$ and $\Haus^{\rho_2}(E)=0$ for all $\rho_2>\sigma$, then $\dim_H E=\sigma$.
\item If $\Haus^\rho(E)=\infty$ for all $\rho<\infty$, then $\dim_H E=\infty$.
\end{itemize}Let $\mu$ be a Borel measure on $\XX$. The \emph{lower Hausdorff dimension} of $\mu$ is defined to be $$\underline{\dim}_H\,\mu = \inf\{\dim_H E: E\subset \XX\text{ is Borel, } \mu(E)>0\}.$$ The \emph{upper Hausdorff dimension} of $\mu$ is defined to be $$\overline{\dim}_H\,\mu = \inf\{\dim_H E : E\subset \XX\text{ is Borel, } \mu(\XX\setminus E)=0\}.$$
\end{definition}

\begin{remark}\label{stable} Both the Hausdorff dimension of a set and upper Hausdorff dimension of a measure are \emph{countably stable}. This means that if $\XX=\bigcup_1^\infty E_i$, then $\dim_H \XX=\sup_i \dim_H E$ and $\overline{\dim}_H\, \mu=\sup_i\overline{\dim}_H\,\mu\res E_i$ for every Borel measure $\mu$ on $\XX$.

Here and below $\mu\res E$ denotes the (outer) measure defined by $\mu\res E(F)=\mu(E\cap F)$ for all $F$. Recall that if $\mu$ is a metric outer measure, then $\mu\res E$ is a metric outer measure for any set $E$; if $\mu$ is a Borel regular outer measure, then $\mu\res E$ is a Borel regular outer measure for any Borel set $E$. See e.g.~\cite{EG} or \cite{Rogers}.\end{remark}

In practice, when working with Hausdorff measures and contents, it is convenient to be able to work with coverings by metric balls or ``cubes'' instead of arbitrary sets.

\begin{lemma}[spherical coverings for null sets] \label{spherical-null-sets} Let $\XX$ be a metric space, let $\rho\in[0,\infty)$, and let $E\subset\XX$. If $\Haus^\rho(E)=0$, then for every $\delta>0$, there exists a sequence $(U(x_i,r_i))_1^\infty$ of open balls in $\XX$ with centers $x_i\in E$ and radii $0<r_i<\delta$ such that $E\subset \bigcup_1^\infty U(x_i,r_i)$ and $\sum_1^\infty (\diam U(x_i,r_i))^\rho <\delta$.\end{lemma}

\begin{proof}Suppose that $\Haus^\rho(E)=0$ and let $\delta>0$ be given. Assign $\epsilon=\min\{\delta,\delta^{1/\rho}\}$. Then $0\leq \Haus^\rho_{\epsilon/4}(E)\leq \Haus^\rho(E)=0$, as well. Since $\Haus^\rho_{\epsilon/4}(E)=0$, there exist sets $E_1,E_2,\dots$ with $\diam E_i\leq \epsilon/4$ such that $E\subset\bigcup_1^\infty E_i$ and $\sum_1^\infty (\diam E_i)^\rho < \delta/2^{1+2\rho}$. Delete any sets $E_i$ that do not intersect $E$ and relabel. For each $i\geq 1$, pick $x_i\in E_i\cap E$ and either set $r_i=2\diam E_i$, if $\diam E_i>0$, or $r_i=\epsilon/2^{1+(i+1)/\rho}$, if $\diam E_i=0$. Note that $r_i<\delta$ for each $i$, since $2\diam E_i \leq \epsilon/2< \delta$ and $\epsilon/2^{1+(i+1)/\rho}<\epsilon/2<\delta$. Further, since each set $E_i\subset U(x_i,r_i)$, we have $E\subset\bigcup_1^\infty U(x_i,r_i)$. Finally, write $I=\{i\geq 1: \diam E_i>0\}$ and $J=\{j\geq 1: \diam E_j=0\}$. Then, by our choices above, \begin{align*}\sum_1^\infty (\diam U(x_i,r_i))^\rho &\leq \sum_{i\in I} 4^\rho (\diam E_i)^\rho + \sum_{j\in J} \frac{\epsilon^\rho}{2^{j+1}}\\ & \leq 4^\rho \sum_1^\infty (\diam E_i)^\rho + \sum_2^\infty \frac{\delta}{2^j}<\frac{\delta}{2}+\frac{\delta}{2}=\delta. \qedhere\end{align*}\end{proof}

\subsection{Parabolic space} \label{ss:cubes}

Let us specialize the ambient metric space to our primary setting. Let $\RR^{n+1}=\RR^n\times\RR$ be equipped with the parabolic distance \eqref{e:p-d}. For each integer $m\geq 2$, we define the system $\Delta^m(\RR^{n+1})$ of \emph{half-open parabolic $m$-adic cubes} to be all sets $Q$ of the form $$Q=\left[\frac{j_1}{m^k},\frac{j_1+1}{m^k}\right)
\times \cdots \times \left[\frac{j_n}{m^k},\frac{j_{n}+1}{m^k}\right)\times \left[\frac{j_{n+1}}{m^{2k}},\frac{j_{n+1}+1}{m^{2k}}\right)\quad(k,j_1,\dots,j_{n},j_{n+1}\in\ZZ);$$ we say that $Q$ belongs to \emph{generation $k$} of $\Delta^m(\RR^{n+1})$ and has \emph{side length} $\side Q=m^{-k}$ and \emph{volume} $\vol Q=m^{-k(n+2)}$. Note that $\diam Q=\sqrt{n} \side Q$ for every $Q\in\Delta^m(\RR^{n+1})$, because we view $\RR^{n+1}$ with the parabolic distance. See Figure \ref{fig:cubes}.

\begin{figure}\begin{center}\includegraphics[width=.3\textwidth]{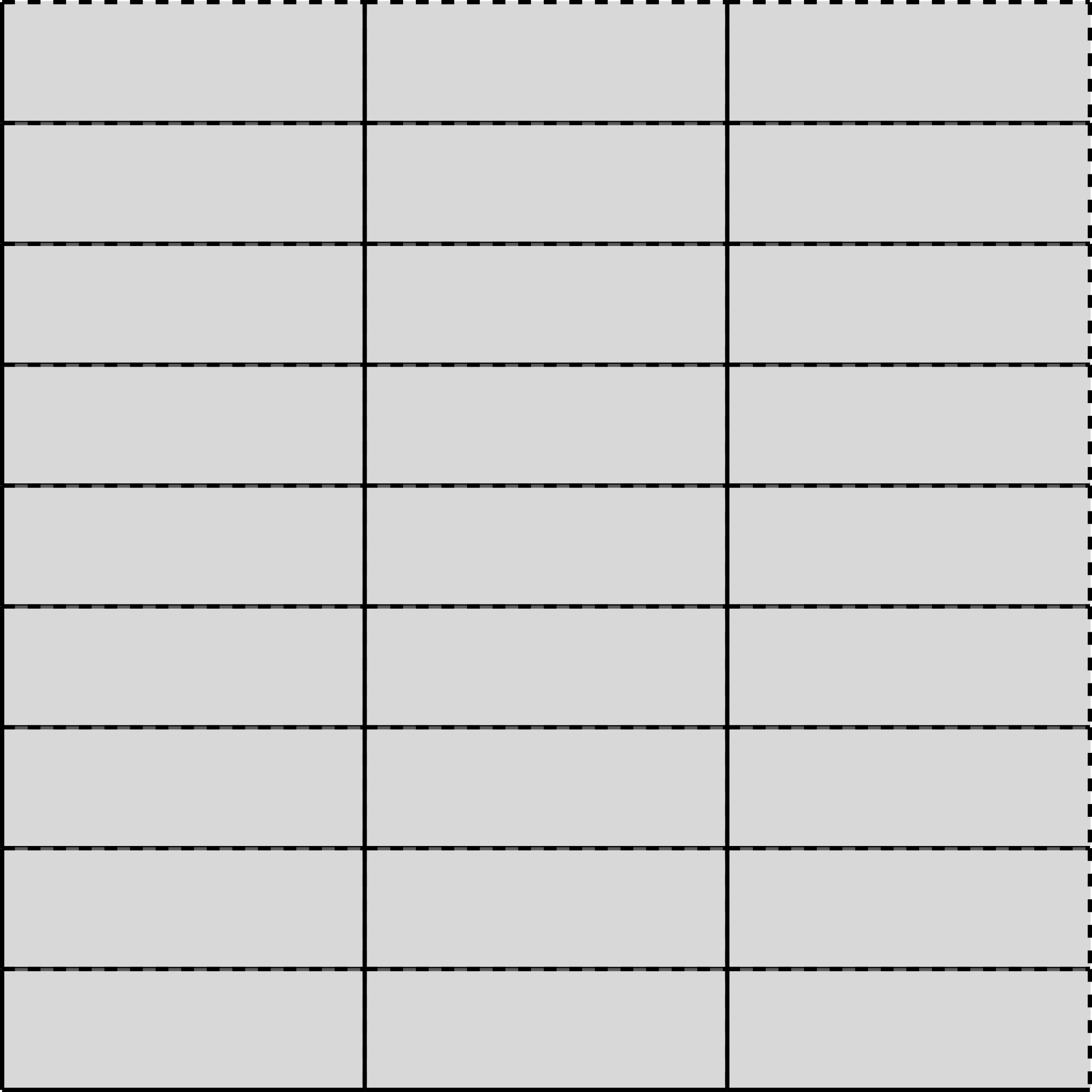}\end{center}
\caption{A parabolic triadic square ($m=3$) in $\RR^2$ of side length 1 and its 27 children of side length 1/3.} \label{fig:cubes}
\end{figure}

The cubes in each generation of $\Delta^m(\RR^{n+1})$ partition $\RR^{n+1}$ and admit an apparent \emph{parent-child hierarchy} or \emph{tree structure}. Every cube $Q\in\Delta^m(\RR^{n+1})$ of generation $k$ is contained in a unique cube $Q^\uparrow\in\Delta^m(\RR^{n+1})$ of generation $k-1$; we call $Q^\uparrow$ the \emph{parent} of $Q$ and we say that $Q$ is a \emph{child} of $Q^\uparrow$. Extending this metaphor, the \emph{ancestors} of a cube include its parent, the parent of its parent, and so on; the \emph{descendents} of a cube include its children, the children of its children, and so on. Note that for every $Q\in\Delta^m(\RR^{n+1})$, $$\#\Child(Q)=m^{n+2}\quad\text{and}\quad \vol Q=\sum_{R\in\Child(Q)}\vol R,$$ where $\Child(Q)$ denotes the set of all children of $Q$. Another observation is that for each cube $Q_0\in\Delta^m(\RR^{n+1})$, there are precisely $3^{n+1}$ cubes $Q\in\Delta^m(\RR^{n+1})$ such that $\side Q=\side Q_0$ and $\overline{Q}\cap\overline{Q_0}\neq\emptyset$ and this number of cubes is independent of $m$.

When discussing the relative position of cubes in a tree, we shall say that an ancestor of a cube is ``above'' the cube and a descendent of a cube is ``below'' the cube.

We now define a family of Hausdorff type measures adapted to the system $\Delta^m(\RR^{n+1})$ of parabolic $m$-adic cubes. Such measures are sometimes called \emph{net measures}; see e.g.~\cite[Ch.~2, \S7]{Rogers}.
The net measures $\Net^\rho$ and net contents $\Net^\rho_\infty$ depend on a choice of a system of cubes, but we suppress this dependence from the notation for the sake of readability. We choose the letter ``$\Net$'' to suggest ``$m$-adic''.

\begin{definition}Fix $\Delta=\Delta^m(\RR^{n+1})$ and let $\rho\in[0,\infty)$. For all $\delta\in(0,\infty]$, we define the set functions $\Net^\rho_\delta:\mathcal{P}(\RR^{n+1})\rightarrow[0,\infty]$ by $$\Net^\rho_\delta(E)=\inf\left\{\sum_1^\infty (\side E_i)^\rho:E\subseteq\bigcup_1^\infty E_i, \side E_i\leq \delta, E_i\in\Delta\right\}\quad\text{for all }E\subset\RR^{n+1};$$ we call $\Net^\rho_\infty$ the \emph{$\rho$-dimensional parabolic net content} on $\RR^{n+1}$. We define the \emph{$\rho$-dimensional parabolic net measure} $\Net^\rho:\mathcal{P}(\RR^{n+1})\rightarrow[0,\infty]$ on $\RR^{n+1}$ by $\Net^\rho(E)=\lim_{\delta\downarrow 0} \Net^\rho_\delta(E)$ for all $E\subset\RR^{n+1}$.\end{definition}

\begin{remark}\label{hausdorff-net} The net contents $\Net^\rho_\infty$ are outer measures and the net measures $\Net^\rho$ are Borel regular outer measures on $\RR^{n+1}$. Unlike the Hausdorff measures and contents, the net measures and contents on $\RR^{n+1}$ are not translation invariant. Nevertheless, for all $n\geq 1$ and $\rho>0$, we have $\Haus^\rho\approx \Net^\rho$ and  $\Haus^\rho_\infty\approx \Net^\rho_\infty$, with implicit constants in one direction depending only on $n$ and in the other direction only on $n$ and $m$. More precisely, $$\Haus^\rho_{\sqrt{n} \delta} \leq \sqrt{n}^\rho \Net^\rho_\delta \leq \sqrt{n}^{n+2} \Net^\rho_\delta,$$ because $\diam Q=\sqrt{n} \side Q$ for every parabolic $m$-cube and $\Haus^\rho_{\sqrt{n}\delta}$ is defined by taking the infimum over a larger class of sets. (Also, all quantities vanish whenever $\rho>n+2$.) Conversely, whenever $\delta=\infty$ or $\delta=m^{-k_0}$ for some $k_0\in\ZZ$, $$\Net^\rho_\delta \leq 3^{n+1}m^\rho\, \Haus^\rho_{\delta}\leq 3^{n+1}m^{n+2}\,\Haus^\rho_\delta.$$ To see this, suppose $E\subset\RR^{n+1}$ is a nonempty set with $m^{-(k+1)}<\diam E\leq m^{-k}$ for some $k\in\ZZ$. Pick any $Q_E\in\Delta$ with $\side Q_E= m^{-k}$ such that $Q_E\cap E\neq\emptyset$. Then $E$ is covered by the $3^{n+1}$ half-open cubes $Q\in\Delta$ with $\side Q=m^{-k}\leq m\diam E$ such that $\overline{Q}\cap\overline{Q_E}\neq\emptyset$. Considering possible coverings for a cube centered at the origin of side length $m^{1/2}$, i.e.~ $$E=\left[-\textstyle\frac{1}{2}m^{1/2},\frac{1}{2}m^{1/2}\right)^n\times\left[-\textstyle\frac{1}{2}m^{-1/4},\frac{1}{2}m^{1/4}\right),$$ shows that some dependence on $m$ in the comparison $\Net^{\rho}_\infty\lesssim \Haus^{\rho}_\infty$ is unavoidable (unless we restrict to $n+2-\rho\ll 1$ depending on $m$).
\end{remark}

The following theorem is usually stated with $m=2$ and the bound $\mu(E)\geq c(n)\,\Haus^\rho_\infty(E)$, obscuring the dependence on $m$ in \eqref{frostman-A}. However, below it is convenient to take $m\gg 2$ and use the sharper bound $\mu(E)\geq c(n)\,\Net^\rho_\infty(E)$; see Lemmas \ref{alternative} and \ref{lemma2}. We refer the reader who may be unfamiliar with the definition of Souslin sets to \cite{Rogers}.

\begin{theorem}[Frostman's lemma]\label{frostman} Fix $\Delta=\Delta^m(\RR^{n+1})$. Let $E\subset\RR^{n+1}$ be a Souslin set (e.g.~any Borel set). If $E$ is bounded and $\Haus^\rho(E)>0$ for some $\rho>0$, then there exists a compact set $K\subset E$ and a finite Borel measure $\mu$ on $\RR^{n+1}$ with support in $K$ such that \begin{equation}\label{frostman-Q}\mu(Q)\leq (\side Q)^\rho\quad\text{for all $Q\in\Delta$},\end{equation} \begin{equation}\label{frostman-c}\mu(E)=\mu(K)\geq c\,\Net^\rho_\infty(E)\geq c\sqrt{n}^{-(n+2)}\,\Haus^\rho_\infty(E),\end{equation}  where $0<c<1$ is a constant depending only on $n$ (and not on $m$!). Furthermore, \begin{equation}\label{frostman-A}\mu(A) \leq 3^{n+1}m^{\rho} (\diam A)^\rho\quad\text{for all sets $A\subset\RR^{n+1}$.} \end{equation} \end{theorem}

\begin{proof}The proof for compact sets in $\RR^n$ with the Euclidean distance using dyadic cubes starting on \cite[p.~112]{Mattila} transfers to the setting of $\RR^{n+1}$ with the parabolic distance without any difficulty. Use parabolic $m$-adic cubes instead of Euclidean dyadic cubes. Although the statement of Frostman's lemma in \cite{Mattila} involves $\Haus^\rho_\infty$, the proof for $E$ contained in some $Q\in\Delta$ only uses coverings by $m$-adic cubes, and thus, one can write $\Net^\rho_\infty$ in the first inequality of \eqref{frostman-c} as indicated (with $c=2^{-(n+1)}$). The second inequality in \eqref{frostman-c} follows from Remark \ref{hausdorff-net}. More generally, if we only know that $E\subset Q$ is Souslin, then we can first choose a compact set $K\subset E$ such that $\Net^\rho_\infty(K)\geq (1/2) \Net^\rho_\infty(E)$; e.g.~combine Corollary 3 (to Theorem 48) and Theorem 52 in \cite{Rogers}. Then we can take the Frostman measure for $E$ to be the Frostman measure for $K$. This requires that we replace $c$ by $c/2$. In general, $E$ may not be contained in a single cube $Q\in\Delta$, but it is contained in $3^{n+1}$ or fewer pairwise disjoint cubes $Q_i\in\Delta$ with $\side Q_i\geq \diam E$ and $\Net^\rho_\infty(E)\leq\sum_i \Net^\rho_\infty (E\cap Q_i)$. If $\mu_i$ is a Frostman measure for $E\cap Q_i$ with constant $c$ for each $i$, then $\mu = 3^{-(n+1)}\sum_i \mu_i$ is a Frostman measure for $E$ with constant $3^{-(n+1)}c$. Finally, \eqref{frostman-A} follows from \eqref{frostman-Q} and Remark \ref{hausdorff-net}.\end{proof}

The following statement gives a method to bound the upper Hausdorff dimension of a measure from above.

\begin{theorem}\label{dim-lemma} Fix $\Delta=\Delta^m(\RR^{n+1})$. Let $\mu$ be a Radon measure on $\RR^{n+1}$ and let $E\subset\RR^{n+1}$ be a Borel set with $\mu(\RR^{n+1}\setminus E)=0$. If there exist constants $0<\rho<n+2$ and $\lambda>0$ such that for every $Q\in\Delta$ with $\side Q\leq 1$, $$\Net^{n+2-\rho}_{m^{-1}\side Q}(E\cap Q)<(\side Q)^{n+2-\rho}\quad\text{or}$$ $$\sum_{R\in\Child(Q)} \mu(R)^{1/2} (\vol R)^{1/2} \leq m^{-\lambda}\, \mu(Q)^{1/2}(\vol Q)^{1/2},$$ then the upper parabolic Hausdorff dimension of $\mu$ is at most $n+2-\lambda\rho/(\lambda+\rho)$.
\end{theorem}

\begin{proof} We model the proof on the argument in \cite[\S2]{Bourgain}, which yields $\overline{\dim}_H\,\mu<n+2$, but keep better track of parameters in order to achieve the sharper conclusion. Because $$\overline{\dim}_H\, \mu = \sup\{\overline{\dim}_H\, \mu\res Q:Q\in\Delta,\ \side Q=1\},$$ we may assume without loss of generality that $E\subset Q_0$ for some $Q_0\in\Delta$ with $\side Q_0=1$. Furthermore, scaling $\mu$, we may assume without loss of generality that $\mu(Q_0)\leq 1$.

Fix $0<\epsilon<1$. We will prove that $$\overline{\dim}_H\,\mu\leq \max\{n+2-\rho(1-\epsilon),n+2-\lambda\epsilon\}.$$ This suffices, because for $\epsilon=\rho/(\lambda+\rho)$, we have $$n+2-\rho(1-\epsilon)=n+2-\lambda\epsilon=n+2-\lambda\rho/(\lambda+\rho).$$
Towards the goal, fix $0<\delta<1$ sufficiently small such that $\epsilon \log_m(1/\delta)\geq 1$ and assign $s=\lceil \epsilon\log_m(1/\delta)\rceil$. We now build a tree $\mathcal{T}=\bigcup_{l=0}^{L} \mathcal{T}_l$ of finite depth $L\leq \lceil \log_m(1/\delta)\rceil$, inductively. Each set in $\mathcal{T}$ is a parabolic $m$-adic cube contained in $Q_0$. Note that while each cube $Q\in\mathcal{T}_{l+1}$ is contained in a unique cube $P\in\mathcal{T}_l$ with $\side P>\side Q$, we shall not specify side lengths of children in $\mathcal{T}$. In other words, the children of a cube $Q$ in $\mathcal{T}$ are descendents of $Q$ in $\Delta$, but are not necessarily children of $Q$ in $\Delta$.

\emph{Base Case.} At the top level, put $\mathcal{T}_0=\{Q_0\}$.

\emph{Induction Step.}  Suppose that $\mathcal{T}_l$ is nonempty for some level $l\geq 0$ and let $Q\in\mathcal{T}_l$. If $\side Q\leq \delta$, then we stop and declare $Q$ to be terminal in $\mathcal{T}$. Otherwise, we shall include at least one and possibly $\aleph_0$ many descendents of $Q$ in $\Delta$ as sets $\mathcal{T}_{l+1}$. If it occurs that  \begin{equation}\label{e:type1} \Net^{n+2-\rho}_{m^{-1}\side Q}(E\cap Q)<(\side Q)^{n+2-\rho},\end{equation} then we say that $Q$ is \emph{type 1}. Choose a finite or countably infinite list $Q^1,Q^2,\dots\in\Delta$ of pairwise disjoint cubes such that $\side Q^i\leq m^{-1}\side Q$ for all $i$, $E\cap Q\subset \bigcup_i Q^i$, and \begin{equation}\label{e:type1-1}\sum_i (\side Q^i)^{n+2-\rho} < (\side Q)^{n+2-\rho}.\end{equation} We include the cube $Q^i\in\mathcal{T}_{l+1}$ for each $i$. According to the hypothesis, if the cube $Q$ is not type 1, then \begin{equation}\label{e:type2}\sum_{R\in\Child(Q)} \mu(R)^{1/2} (\vol R)^{1/2} \leq m^{-\lambda}\, \mu(Q)^{1/2}(\vol Q)^{1/2},\end{equation} where $\Child(Q)$ denotes the set of all children of $Q$ in $\Delta$. In this instance, we say that $Q$ is \emph{type 2}. Include each $R\in\Child(Q)$ in $\mathcal{T}_{l+1}$. Repeating this procedure for each $Q\in\mathcal{T}_l$ completes the definition of $\mathcal{T}_{l+1}$.

To see that $\mathcal{T}$ has finite depth, simply note that if $Q\in\mathcal{T}_l$, $R\in\mathcal{T}_{l+1}$, and $R\subset Q$, then $\side R\leq m^{-1} \side Q$. If $\mathcal{T}_{\lceil \log_m(1/\delta)\rceil}$ is nonempty, then $\side Q\leq m^{-\log_m(1/\delta)}\side Q_0\leq \delta$ and $Q$ is terminal in $\mathcal{T}$ for every $Q\in\mathcal{T}_{\lceil \log_m(1/\delta)\rceil}$. Thus, $\mathcal{T}_{\lceil \log_m(1/\delta)\rceil+1}$ is never defined.

For later use, note that for any $Q\in\mathcal{T}$ (in particular, for cubes of type 2), \begin{equation}\label{e:type2-2}\Net^{n+2-\rho}_{m^{-1}\side Q}(E\cap Q)
\leq \sum_{R\in\Child(Q)} (\side R)^{n+2-\rho} = m^{\rho} (\side Q)^{n+2-\rho}.
\end{equation} However, whenever $Q\in\mathcal{T}_l$ is type 1, equation \eqref{e:type1-1} yields the improved estimate \begin{equation}\label{e:type1-2}\Net^{n+2-\rho}_{m^{-1}\side Q}(E\cap Q)\leq \sum_i (\side Q^i)^{n+2-\rho}< (\side Q)^{n+2-\rho}.\end{equation}

\emph{Goal.} Write $\eta=\max\{n+2-\rho(1-\epsilon),n+2-\lambda\epsilon\}$. With the tree $\mathcal{T}$ in hand, we will shortly identify a measurable subset $F_\delta\subset E$ such that \begin{equation}\label{M-goal} \mu(\RR^{n+1}\setminus F_{\delta})=\mu(E\setminus F_\delta)=o(\delta)\quad\text{and}\quad\Net^{\eta+\alpha}_{m\,\delta^{\epsilon}}(F_\delta)\leq 2m^\alpha\delta^{\epsilon\alpha}\quad\text{for all }\alpha\geq 0.\end{equation}
To proceed, let $\mathcal{T}^*$ denote the set of all terminal cubes in $\mathcal{T}$. Observe that $\mathcal{T}^*\subset \bigcup_1^L \mathcal{T}_l$, because $Q_0$ is not terminal as $\delta<1=\side Q_0$. We now use the terminal cubes to define two collections $\mathcal{E},\mathcal{F}\subset\mathcal{T}$ with $\mathcal{E}\cap\mathcal{F}=\emptyset$. We say that $Q\in \mathcal{T}^*\cap \mathcal{T}_l$ is \emph{efficient for $\Net^{n+2-\rho}$} and include $Q\in\mathcal{E}$ if the chain $Q_0,\dots,Q_{l-1}$ of cubes strictly above $Q$ in $\mathcal{T}$ contains fewer than $s=\lceil \epsilon \log_m(1/\delta)\rceil$ type 2 cubes. Otherwise, if $Q\in\mathcal{T}^*$ is not efficient for $\Net^{n+2-\rho}$, then there exists an ancestor $P$ of $Q$ in $\mathcal{T}$ such that $P$ is the $s$-th type 2 cube appearing in the chain from $Q_0$ to $Q$, starting from $Q_0$; in this case, we include $P$ in the collection $\mathcal{F}$. By construction of the tree, we have $E\subset \bigcup\mathcal{T}^*\subset \bigcup\mathcal{E} \cup \bigcup \mathcal{F}$. Furthermore, $\side Q\leq \delta$ for all $Q\in\mathcal{E}$ and $\delta<\side Q\leq m^{-(s-1)} \leq m \delta^\epsilon$ for all $Q\in\mathcal{F}$. Thus, $\mathcal{E}\cup\mathcal{F}$ is a cover of $E$ by parabolic $m$-cubes of side length at most $m\,\delta^\epsilon$. Define $\mathcal{G}$ to be the set of all $Q\in\mathcal{F}$ such that $(\side Q)^{n+2-\lambda\epsilon} \leq \mu(Q)$ and put $F_\delta = E\cap \bigcup(\mathcal{E}\cup \mathcal{G})$.

We now make a series of estimates towards \eqref{M-goal}. Recall that at most $s-1\leq \epsilon\log_m(1/\delta)$ ancestors of cubes in $\mathcal{E}$ are type 2. By \eqref{e:type2-2} and \eqref{e:type1-2}, summing layer-by-layer up the tree from $\mathcal{E}$ to $Q_0$, \begin{equation*}\sum_{Q\in\mathcal{E}} (\side Q)^{n+2-\rho}\leq \cdots \leq (m^\rho)^{s-1} (\side Q_0)^{n+2-\rho} \leq \delta^{-\rho\epsilon}. \end{equation*} Hence, since $\side Q\leq \delta<1$ for all $Q\in\mathcal{E}$, we have that for any $\alpha\geq 0$, $$\Net^{\eta+\alpha}_{m\,\delta^\epsilon}(E\cap \bigcup\mathcal{E}) \leq \sum_{Q\in\mathcal{E}} (\side Q)^{n+2-\rho(1-\epsilon)+\alpha}
\leq \delta^{\rho\epsilon+\alpha}\sum_{Q\in\mathcal{E}} (\side Q)^{n-2-\rho} \leq \delta^\alpha.$$ The estimate for $E\cap\bigcup\mathcal{G}$ is even easier. By definition of $\mathcal{G}$, for any $\alpha\geq 0$,
$$\Net^{\eta+\alpha}_{m\,\delta^\epsilon}(E\cap \bigcup\mathcal{G}) \leq \sum_{Q\in\mathcal{G}} (\side Q)^{n+2-\lambda\epsilon+\alpha} \leq m^\alpha\delta^{\epsilon\alpha}\sum_{Q\in\mathcal{G}} \mu(Q) \leq m^\alpha\delta^{\epsilon\alpha}.$$ Thus, $\Net^{\eta+\alpha}_{m\,\delta^\epsilon}(F_\delta) \leq \delta^\alpha+m^\alpha\delta^{\epsilon\alpha}\leq 2m^{\alpha}\delta^{\epsilon\alpha}$ for all $\alpha\geq 0$. To verify \eqref{M-goal}, it remains to estimate $\mu(E\setminus F_\delta)$. Observe that $$\vol Q = (\side Q)^{n+2} > (\side Q)^{\lambda\epsilon} \mu(Q) > \delta^{\lambda\epsilon} \mu(Q)\quad\text{for all }Q\in\mathcal{F}\setminus\mathcal{G}.$$ Now, by the Cauchy-Schwarz inequality, $$\sum_{\{R\}} \mu(R)^{1/2}(\vol R)^{1/2}\leq \mu(Q)^{1/2}(\vol Q)^{1/2}$$ whenever $Q$ is partitioned by cubes $\{R\}$. However, we have the improved estimate \eqref{e:type2} whenever $Q\in\mathcal{T}$ is type 2. As every $Q\in\mathcal{F}$ has $(s-1)$ ancestors in $\mathcal{T}$ of type 2, we get\begin{align*}
\mu(E\setminus F_\delta) \leq \sum_{Q\in\mathcal{F}\setminus\mathcal{G}} \mu(Q)&< \delta^{-\lambda\epsilon/2}\sum_{Q\in\mathcal{F}\setminus\mathcal{G}} \mu(Q)^{1/2}(\vol Q)^{1/2} \\ &\leq \delta^{-\lambda\epsilon/2}m^{-\lambda(s-1)}\mu(Q_0)^{1/2}(\vol Q_0)^{1/2}\leq m^{\lambda}\delta^{\lambda\epsilon/2}=o(\delta),\end{align*} as desired.

\emph{Conclusion.} Choose any sequence $\delta_1> \delta_2>\cdots > 0$ such that $\sum_1^\infty \mu(E\setminus F_{\delta_i}) <\infty$ and put $F=\liminf_{j\rightarrow\infty} F_{\delta_j}=\bigcup_{j=1}^\infty\bigcap_{i=j}^\infty F_{\delta_i}$. By the Borel-Cantelli lemma, the set $E\setminus F=\limsup_{j\rightarrow\infty} (E\setminus F_{\delta_i})$ has $\mu$ measure zero. Finally, $F$ is Borel (since $E$ is Borel) and for any $\alpha>0$, $$\Net^{\eta+\alpha}(F) = \lim_{j\rightarrow\infty} \Net^{\eta+\alpha}\left(\bigcap_{i=j}^\infty F_{\delta_i}\right)
=\lim_{j\rightarrow\infty}\sup_{k\geq j} \Net^{\eta+\alpha}_{m\,\delta_k^\epsilon}\left(\bigcap_{i=j}^\infty F_{\delta_i}\right)
\leq \lim_{j\rightarrow\infty} 2m^{\alpha} \delta_j^{\epsilon\alpha}=0$$ by continuity from below and \eqref{M-goal}. Therefore, $\overline{\dim}_H\,\mu \leq \dim_H F\leq \eta$.
\end{proof}

\begin{remark}At the end of the proof, we could also take $\alpha=0$ and obtain the stronger conclusion $\mathcal{M}^\eta(F)\leq 2<\infty$, whence $\mu$ is carried by (countable unions of) sets of finite $\Haus^\eta$ measured. Of course, this is only interesting if $\overline{\dim}_H\, \mu = \eta$. To that end:\end{remark}

\begin{question}\label{q:sharp} For some/each choice of $\rho$ and $\lambda$, does there exist a measure $\mu$ satisfying the hypothesis of Theorem \ref{dim-lemma} with $\overline{\dim}_H\, \mu = n+2 - \lambda\rho/(\lambda+\rho)$?\end{question}

Finally, we need the following fact in \S\ref{sec:constants}. See \cite[Theorem 8.10]{Mattila} for details.

\begin{lemma}\label{l:products} Let $A\subset\RR^n$ and $B\subset\RR$ (equipped with the snowflaked metric $|x-y|^{1/2}$) be non-empty Borel sets. \begin{enumerate}
\item If $\Haus^s(A)>0$ and $\Haus^t(B)>0$, then $\Haus^{s+t}(A\times B)>0$.
\item In general, $\dim_H (A\times B)\geq \dim_H A + \dim_H B$.
\item If $\dim_H B=2$, then $\dim_H(A\times B)=\dim_H A + 2$.\end{enumerate}\end{lemma}

\section{Preliminaries II: heat potential theory and caloric measure} \label{sec:heat}

In this section, we give a brief synopsis of heat potential theory, which is presented with great clarity in the excellent book of Watson \cite{Watson}. For additional results, the daring may consult the comprehensive research monograph of Doob \cite{Doob}. Recent advances moving parabolic potential theory closer to the time-independent theory are given by Mourgoglou and Puliatti \cite{MP-caloric}. At the end of the section, we give a statement and proof of the strong Markov property for caloric measure, which is unlikely new, but we are not aware of a good reference to it in the literature. The subtle interplay of the essential and singular boundaries of nested domains makes things a bit tricky.

To \emph{formulate} the Dirichlet problem for the heat equation, $-\Delta_X\,u + \partial_t\,u=0$, we must first classify different parts of the topological boundary of a domain in $\RR^{n+1}=\RR^n\times\RR$. Because of the asymmetric flow of time, it is useful to split balls in the parabolic distance into time-backwards and time-forwards halves. For all $(X,t)\in\RR^{n+1}$ and $r>0$, let $$U^-_r(X,t)=U_{\RR^n}(X,r)\times(t-r^2,t)\quad\text{and}\quad U^+_r(X,t)=U_{\RR^n}(X,r)\times(t,t+r^2)$$ denote half balls excluding the present and extending into the past and future, respectively. Here $U_{\RR^n}(X,r)$ denotes the usual open ball in $\RR^n$ with the Euclidean distance. Note that, because the parabolic distance \eqref{e:p-d} is defined using the supremum norm, each $U^\pm_r(X,t)$ has the shape of a \emph{cylinder} in the time-axis whose \emph{base} is a ball in the space coordinates.

\begin{definition}[{classification of boundary points}]\label{parabolicb} Let $\Omega\subsetneq\RR^{n+1}$ be a nonempty open set and let $\partial\Omega$ denote the topological boundary of $\Omega$ in a one-point compactification of $\RR^{n+1}$ so that $\partial\Omega$ includes the point at infinity, which we denote by $\infty$, if $\Omega$ is unbounded. Following \cite[\S8.1]{Watson}, we partition $\partial\Omega$ as $$\partial\Omega=\partial_e\Omega\cup\partial_s\Omega=(\partial_n\Omega\cup\partial_{ss}\Omega)\cup\partial_s\Omega,$$ where $(X,t)\in \partial\Omega$ belongs to \begin{itemize}
\item the \emph{normal boundary} $\partial_n\Omega$ if either $(X,t)=\infty$ or $U^-_r(X,t)\setminus \Omega\neq\emptyset$ for every $r>0$;
\item the \emph{semi-singular boundary} $\partial_{ss}\Omega$ if $U^-_{r_0}(X,t)\subset \Omega$ for some $r_0>0$ \emph{and} also $U^+_{r_1}(X,t)\cap \Omega\neq\emptyset$ for every $0<r_1<r_0$; and,
\item the \emph{singular boundary} $\partial_s\Omega$ if $U^-_{r_0}(X,t)\subset \Omega$ for some $r_0>0$ \emph{and} $U^+_{r_1}(X,t)\cap \Omega=\emptyset$ for some $0<r_1<r_0$.
\end{itemize} The \emph{essential boundary} $\partial_e\Omega$ is the union of the normal and semi-singular boundaries.
\end{definition}

That is to say, the normal boundary consists of those points, which see the complement of $\Omega$ in the past. In the \emph{abnormal boundary} $\partial\Omega\setminus\partial_n\Omega$, the semi-singular boundary points are those that can be approached by a sequence of points in $\Omega$ from the future, whereas the singular boundary points cannot. See Figure \ref{fig:boundary}. For any domain $\Omega$, there exists a countable set $T$ of times such that $\partial_{ss}\Omega\cup\partial_s\Omega\subset \RR^n\times T$ (see \cite[Theorem 8.40]{Watson}). \begin{figure}\begin{center}\includegraphics[width=.47\textwidth]{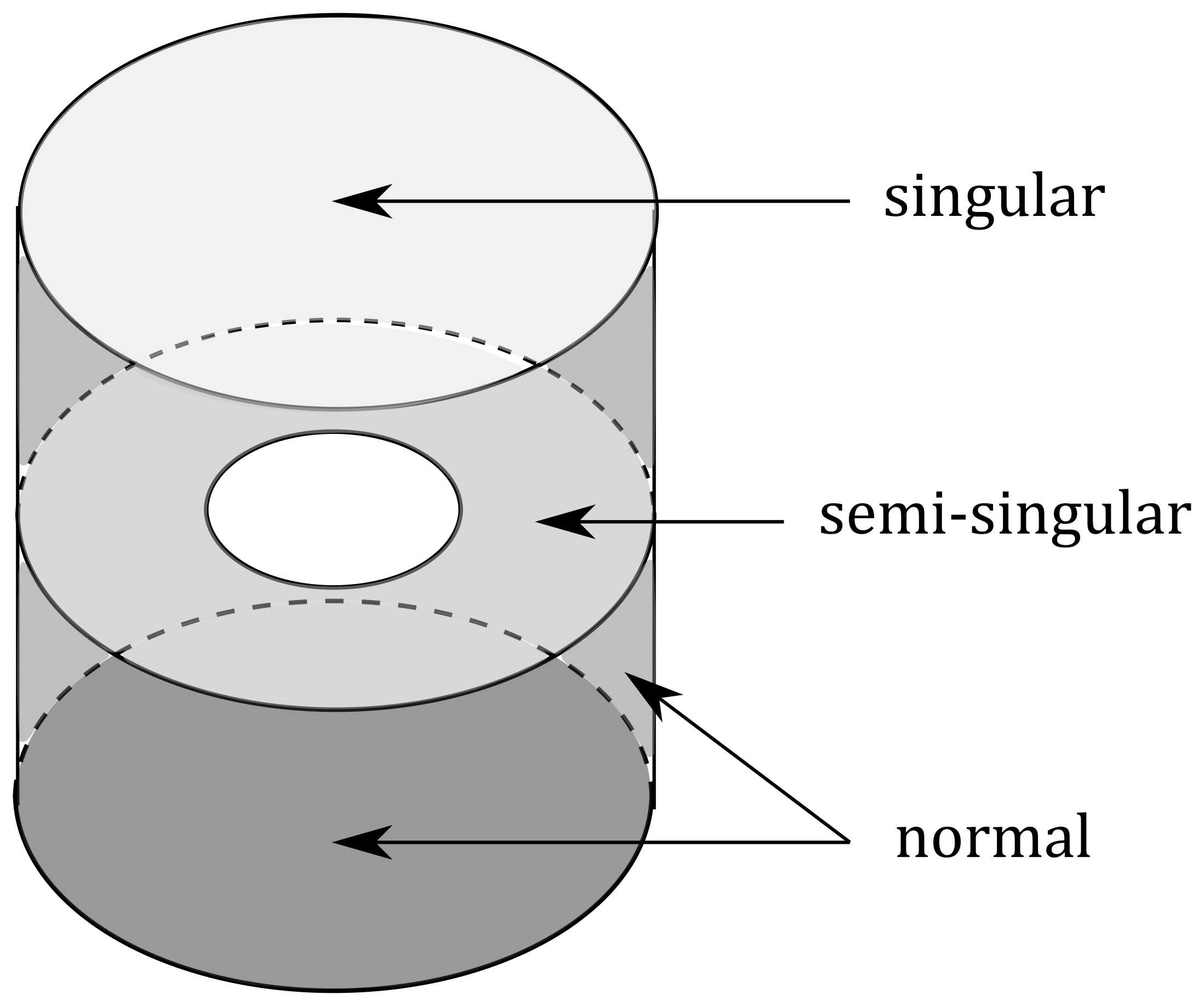}\end{center}\caption{The normal, semi-singular, and singular boundaries of two stacked open cylinders in $\RR^2\times\RR$, connected by a circular cut in the center. Time is oriented vertically, flowing upward.}\label{fig:boundary}\end{figure}

\begin{remark} The definitions of $\partial_n\Omega$, $\partial_{ss}\Omega$, and $\partial_s\Omega$ in \cite{Watson} uses time-backwards and time-forwards half balls in $\RR^{n+1}$ equipped with the Euclidean metric instead of $U^-_r(X,t)$ and $U^+_r(X,t)$. However, as any Euclidean half-ball is contained in a parabolic half-ball and vice-versa, the two definitions are equivalent.\end{remark}

The \emph{Dirichlet problem for the heat equation} on a domain (i.e.~a connected, open set) $\Omega\subset\RR^{n+1}$ is to find, for each (bounded) continuous function $f$ on the essential boundary, a \emph{temperature} $u(X,t)$, i.e.~a $C^{2,1}\equiv C_X^2C_t^1$ solution of the heat equation in $\Omega$, such that \begin{equation}\label{e:Dir-conditions}\left\{\begin{array}{l} u(X,t)\rightarrow f(X_0,t_0)\quad\text{as $(X,t)\in\Omega\rightarrow (X_0,t_0)\in\partial_n\Omega$},\\
u(X,t)\rightarrow f(X_0,t_0)\quad\text{as $(X,t)\in\Omega\rightarrow (X_0,t_0)\in\partial_{ss}\Omega$ with $t>t_0$}.\end{array}\right.\end{equation} Note that we only consider limits at points in the semi-singular boundary as they are approached from the future. There are no time restrictions on the approach to points in the normal boundary. In the Dirichlet problem, we never specify the boundary values of a solution along the singular boundary. This is justified by thinking about the flow of heat from the past to the future in the ``stacked cylinder'' domain in the figure.

As with the Dirichlet problem for harmonic functions, the Dirichlet problem for the heat equation does not admit a classical solution on arbitrary domains, because \eqref{e:Dir-conditions} may fail for some $f\in C(\partial_e\Omega)$ at \emph{irregular} boundary points $(X_0,t_0)\in\partial_e\Omega$. On the other hand, if at some $(X_0,t_0)\in\partial_e\Omega$, \eqref{e:Dir-conditions} holds for every $f\in C(\partial_e\Omega)$, then we say that $(X_0,t_0)$ is \emph{regular} for the Dirichlet problem. Regularity of a boundary point is a property that depends on the domain $\Omega$ (actually, on the asymptotic behavior of the thermal capacity of the complement  $\RR^{n+1}\setminus\Omega$ at the boundary point). We refer the reader to \cite[\S8.5]{Watson} and \cite{Lanconelli-Wiener,EG-Wiener}  for details.

It is useful to be able to consider the Dirichlet problem and its solution in the extended sense of Perron-Wiener-Brelot. The description is rather long. As a shortcut, we refer the reader to \cite[\S\S3.2,\,8.2]{Watson} for the definition of hypotemperatures / hypertemperatures and subtemperatures / supertemperatures. Let $\Omega\subsetneq\RR^{n+1}$ be a domain and let $f$ be an extended real-valued function on $\partial_e\Omega$. The \emph{lower class} $\mathfrak{L}^f_\Omega$ consists of all hypotemperatures $u$ on $\Omega$ bounded from above (including $u\equiv -\infty$) such that \begin{equation}\label{e:lower}\left\{\begin{array}{ll} \limsup_{(X,t)\rightarrow(X_0,t_0)} u(X,t)\leq f(X_0,t_0)\quad&\text{for all } (X_0,t_0)\in\partial_n\Omega,\\
\limsup_{(X,t)\rightarrow(X_0,t_0),\,t>t_0} u(X,t)\leq f(X_0,t_0)\quad&\text{for all }(X_0,t_0)\in\partial_{ss}\Omega.\end{array}\right.\end{equation} Similarly, the \emph{upper class} $\mathfrak{U}^f_\Omega$ consists of all hypertemperatures $v$ on $\Omega$ bounded from below (including $v\equiv +\infty$) such that \begin{equation}\label{e:upper}\left\{\begin{array}{ll} \liminf_{(X,t)\rightarrow(X_0,t_0)} v(X,t)\geq f(X_0,t_0)\quad&\text{for all } (X_0,t_0)\in\partial_n\Omega,\\
\liminf_{(X,t)\rightarrow(X_0,t_0),\,t>t_0} v(X,t)\geq f(X_0,t_0)\quad&\text{for all }(X_0,t_0)\in\partial_{ss}\Omega.\end{array}\right.\end{equation} It can be shown that $u\leq v$ on $\Omega$ for all $u\in\mathfrak{L}^f_\Omega$ and $v\in\mathfrak{U}^f_\Omega$. We call $$L^f_\Omega\equiv \sup\mathfrak{L}^f_\Omega\quad\text{and}\quad U^f_\Omega\equiv \inf\mathfrak{U}^f_\Omega$$ the \emph{lower} and \emph{upper solutions} for $f$ on $\Omega$, respectively. If  $L^f_\Omega=U^f_\Omega$ on $\Omega$, then we say that $f$ is \emph{resolutive} and declare $H^f_\Omega\equiv L^f_\Omega=U^f_\Omega$ to be the \emph{PWB solution} to the Dirichlet problem for the heat equation on $\Omega$ with boundary data $f$. For any resolutive function, the PWB solution is a temperature in $\Omega$. Every $f\in C(\partial_e\Omega)$ is resolutive and $u=H^f_\Omega$ satisfies \eqref{e:Dir-conditions} at all regular points $(X_0,t_0)\in\partial_e\Omega$. For details and further results, we refer the reader to \cite{Watson}. All of this preamble leads to...

\begin{theorem}[existence and support of caloric meaure] \label{parabolic}
Let $\Omega\subsetneq \RR^{n+1}$ be a connected, open set. For each $(X,t)\in\Omega$, there exists a unique Borel regular metric outer measure called the \emph{caloric measure} of $\Omega$ with \emph{pole} at $(X,t)$ such that for each function $f\in C(\partial_e\Omega)$,  for each indicator function $f=\chi_E$ of a Borel set $E\subset\partial_e\Omega$, and more generally, for each function with $f\in L^1(\omega^{X,t}_\Omega)$ for all $(X,t)\in\Omega$, the PWB solution $H^f_\Omega$ exists and
\begin{equation}\label{parmeasuredef}
H^f_\Omega(X,t)=\int f\,d\omega^{X,t}_\Omega\quad\text{for all }(X,t)\in\Omega.
\end{equation} For each $(X,t)$, the caloric measure $\omega^{X,t}_\Omega$ is a probability measure carried by the subset of the essential boundary that is accessible by strictly backwards-in-time paths in $\Omega$ starting from the pole. That is, \begin{equation}\label{omega-support}\omega^{X,t}_\Omega(\RR^{n+1})=\omega^{X,t}_\Omega(\partial_e\Omega)=\omega^{X,t}_\Omega(\Gamma^-_\Omega(X,t))=1,\end{equation} where $\Gamma^-_\Omega(X,t)$ denotes the set of all points $(Y,s)\in\RR^{n+1}$ with $s<t$ such that there exists a continuous map $\gamma:[0,t-s)\rightarrow\Omega$ satisfying $\gamma(0)=(X,t)$, $\gamma(\tau)\in\RR^{n}\times\{t-\tau\}$ for all $\tau\in[0,t-s)$, and $\lim_{\tau\rightarrow t-s} \gamma(\tau)=(Y,s)$.\end{theorem}

\begin{proof} See Theorem 8.27, Lemma 8.29, and Theorem 8.32 in \cite{Watson}. The fact that a caloric measure on $\partial\Omega$, if it exists, should be supported in the essential boundary is due to Suzuki \cite{Suzuki}.\end{proof}

\begin{remark}\label{zero-future} The last condition in Theorem \ref{parabolic} guarantees that the caloric measure of the present and the future is zero: $\omega^{X,t}_\Omega(\RR^n\times\{s\in\RR:s\geq t\})=0$ for all $\Omega$ and $(X,t)\in\Omega$. This fact is obvious from the probabilistic interpretation of caloric measure mentioned at the top of the introduction. As temporal beings, we may find the fact that the distribution of heat in the future does not affect us in the present to be reassuring.\end{remark}

\begin{remark}\label{r:harnack} By Harnack's inequality, $\omega^{X_1,t_1}_\Omega\ll\omega^{X,t}_\Omega$ whenever $(X_1,t_1)\in\Omega\cap \Gamma^-_\Omega(X,t)$. See \cite[Corollary 1.33]{Watson}.\end{remark}

\begin{remark} As a warning, especially for readers who are familiar with harmonic measure, we note that there exist domains $\Omega\subset\RR^{n+1}$ such that the set $I_\Omega$ of irregular points in $\partial_e\Omega$ can have \emph{positive} caloric measure (see \cite{TW85}). Such domains have chaotic, rapidly changing time-slices $\Omega_\tau=\Omega\cap(\RR^n\times\{\tau\})$ that make $\Omega$ pathological from the perspective of the Dirichlet problem. On the other hand, be assured that $\omega^{X,t}_\Omega(I_\Omega)=0$ for all $(X,t)\in\Omega$ whenever $\Omega=\Omega_0\times (a,b)$ is a cylinder over a domain $\Omega_0\subsetneq\RR^n$.
\end{remark}

A set $Z\subset\RR^{n+1}$ is said to be \emph{polar} for the heat equation if there exists an open set $\Omega$ containing $Z$ and a supertemperature $v$ on $\Omega$ such that $v(X,t)=\infty$ for every $(X,t)\in Z$. For basic facts about polar sets and related potential theory, see \cite[Chapter 7]{Watson}. Among other results, Taylor and Watson \cite{TW85} prove that (i) \emph{every set $Z\subset\RR^{n+1}$ with $\Haus^n(Z)=0$ is polar}, and also (ii) \emph{every set $Z\subset\RR^{n+1}$ with $\dim_H Z>n$ is not polar}, where as in \S2 Hausdorff measures and dimension are defined using parabolic distance on $\RR^{n+1}$. Polar sets have caloric measure zero in any domain \cite[\S5]{Watson-open}.

\begin{theorem}[maximum principle {\cite[Theorem 8.2]{Watson}}]\label{maximum principle} Let $u$ be a hypotemperature on an open set $\Omega\subset\RR^{n+1}$ and let $Z$ be a polar subset of $\partial_e \Omega$. Suppose that \begin{align}\label{max-1}
\limsup_{(X,t)\to(X_0,t_0)} u(X,t)&<\infty\quad\text{for all }(X_0,t_0)\in \partial_e \Omega\quad \text{and}\\ \label{max-2}
\limsup_{(X,t)\rightarrow(X_0,t_0)} u(X,t)&\leq A\,\quad\text{for all }(X_0,t_0)\in (\partial_e \Omega)\setminus Z,\end{align} where for $(X_0,t_0)\in\partial_{ss}\Omega$ the approach of $(X,t)$ to $(X_0,t_0)$ is along points with $t>t_0$.  Then $u(X,t)\leq A$ for all $(X,t)\in \Omega$.
\end{theorem}

We say that a pointwise defined property $P(X,t)$ holds \emph{nearly everywhere} if the property holds for all points in the complement of a polar set.

\begin{theorem}[strong Markov property for caloric measure] \label{t:strong-markov}
Let $\Omega_1\subset \Omega_2 \subsetneq \RR^{n+1}$ be nested domains. In addition, suppose that \begin{enumerate}
\item nearly every point in $\partial_e\Omega_1$ is regular for the Dirichlet problem on $\Omega_1$,
\item nearly every point in $\partial_e\Omega_1\cap \partial_e\Omega_2$ is regular for the Dirichlet problem on $\Omega_2$, and
\item the sets $\partial_e\Omega_1\cap \partial_s\Omega_2$ and $\partial_n\Omega_1\cap \partial_{ss}\Omega_2$ are polar.
\end{enumerate} If $E\subset\RR^{n+1}$ is a Borel set and $(X_1,t_1)\in \Omega_1$, then
\begin{equation}\label{StrongMarkov}
    \omega_{\Omega_2}^{X_1,t_1}(E)=\omega_{\Omega_1}^{X_1,t_1}(E\cap\partial_e\Omega_2)+ \int_{(\partial_e \Omega_1)\cap \Omega_2} \omega_{\Omega_2}^{Z,\tau}(E)\, d\omega_{\Omega_1}^{X_1,t_1}(Z,\tau).
\end{equation}
\end{theorem}

\begin{proof} We model our argument on the analytic proof of the strong Markov property for elliptic measures presented in \cite[Lemma B.1]{AAM19}. By the inner regularity of Radon measures on their measurable sets and separability of $\RR^{n+1}$, it suffices to prove \eqref{StrongMarkov} for compact sets in $\partial_e\Omega_2$. (See the argument at the end of the proof in \cite{AAM19} for details.)

\begin{figure}\begin{center}\includegraphics[width=.7\textwidth]{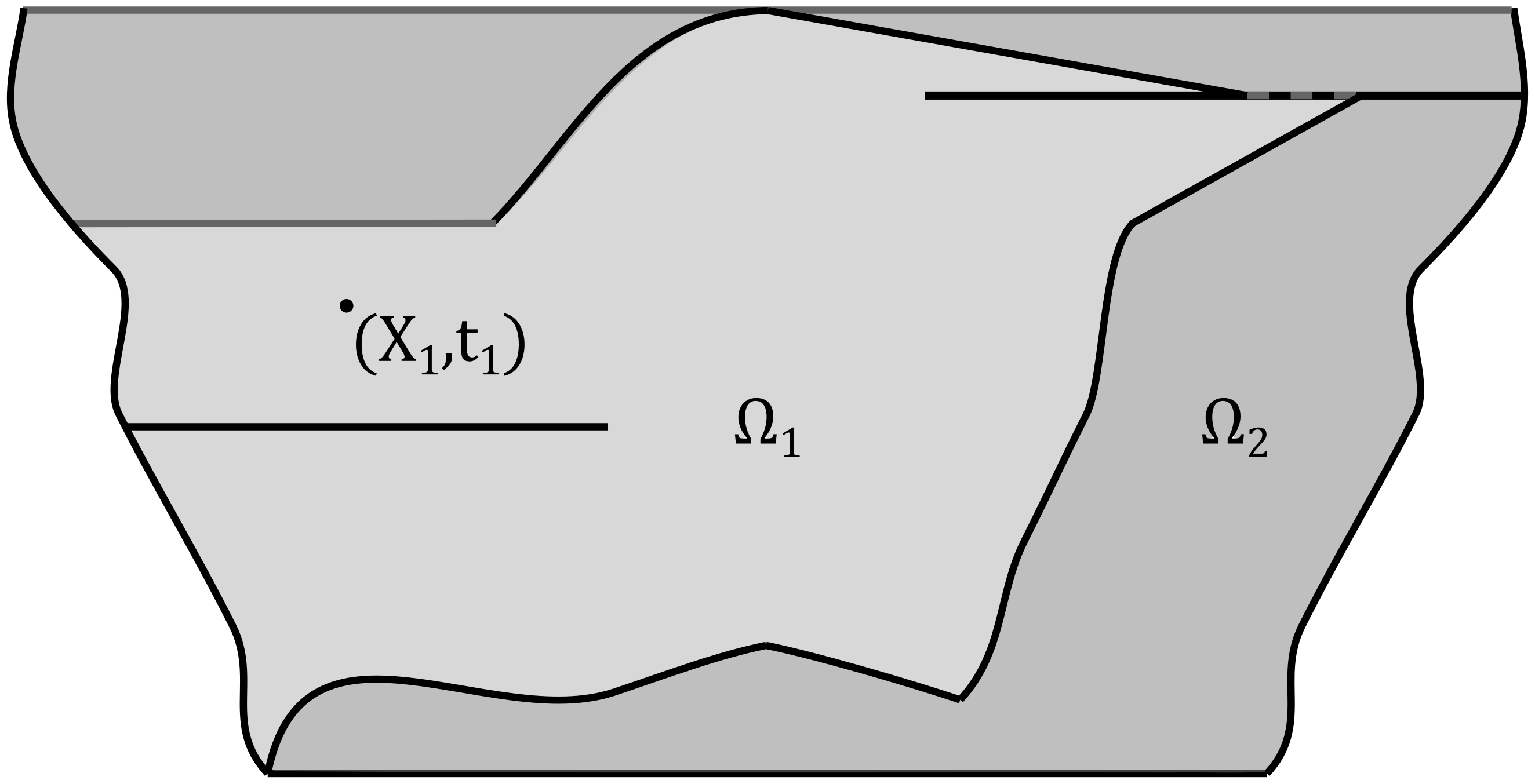}\end{center}
\caption{Nested domains $\Omega_1\subset\Omega_2$ in $\RR^1\times\RR$ with essential boundaries indicated by black lines and singular boundaries by gray lines. At the top, we illustrate how $\partial_e\Omega_1\cap \partial_s\Omega_2\neq\emptyset$ and $\partial_s\Omega_1\cap\partial_{e}\Omega_2\neq 0$ are both possible. The pole $(X_1,t_1)$ lies in both domains. } \label{fig:markov}
\end{figure}

Let $E\subset \partial_e\Omega_2$ be compact. Then we may choose a sequence $\phi_j\in C_c(\partial_e\Omega_2)$ of real-valued continuous functions with compact support such that $0\leq \phi_j\leq 1$ and $\phi_j(X,t)\downarrow \chi_E(X,t)$ for all $(X,t)\in \partial_e\Omega_2$. Consider the temperatures $u_j$ in $\Omega_2$ defined by
$$u_j(X,t)=\int_{\partial_e\Omega_2}\phi_j\, d\omega_{\Omega_2}^{X,t}\quad\text{for all }(X,t)\in\Omega_2.$$
To proceed, let us partition $$\partial_e\Omega_1=(\partial_e\Omega_1\cap \Omega_2)\cup(\partial_e\Omega_1\cap\partial_e\Omega_2)\cup(\partial_e\Omega_1\cap\partial_s\Omega_2)=:A_1\cup B_1\cup C_1.$$ In Figure \ref{fig:markov}, we illustrate an example of nested domains where $A_1$, $B_1$, and $C_1$ are each nonempty. Let $\overline{u}_j$ denote an extension of $u_j$ to $\Omega_2\cup B_1\cup C_1$ defined by setting
$$\overline{u}_j=\left\{ \begin{array}{ll}u_j\quad &\text{in }\Omega_2 \\
\phi_j\quad &\text{in }B_1,\\
\,0\quad  &\text{in }C_1.
\end{array} \right.$$
Then $\overline{u}_j$ is continuous everywhere in $\Omega_2$ (hence on $A_1$), simply because $u_j$ is a temperature. In addition, $\overline{u}_j$ is continuous at nearly every point in $B_1$ by (ii) and continuity of $\phi_j$, where as usual if $(X_0,t_0)\in B_1\cap\partial_{ss}\Omega_2$, then continuity of $\overline{u}_j$ at $(X_0,t_0)$ means on approach by points $(X,t)$ with $t>t_0$. Unfortunately, it could be that $(X_0,t_0)\in \partial_n\Omega_1\cap\partial_{ss}\Omega_2$ (e.g.~imagine $\partial_n\Omega_1$ includes a segment $\{X_0\}\times[t_0-\epsilon,t_0)$ in $\Omega_2$ that is backwards-in-time from $(X_0,t_0)\in\partial_{ss}\Omega_2$), in which case we have no way to guarantee that $\overline{u}_j|_{\Omega_1}$ is continuous at $(X_0,t_0)$ when approached from the past. We deal with this possibility in assumption (iii) by requiring that $\partial_n\Omega_1\cap\partial_{ss}\Omega_2$ be a polar set. We do not worry about continuity of the extension $\overline{u}_j$ at points in $C_1$, since anyways $C_1$ is also polar by (iii). As $\overline{u}_j$ is defined on all of $\partial_e\Omega_1$, we may define a temperature $v_j$ on $\Omega_1$ by assigning
$$v_j(X,t)=\int_{\partial_e\Omega_1} \overline{u}_j\,d\omega_{\Omega_1}^{X,t}\quad\text{for all }(X,t)\in\Omega_1.$$ Observe that (see \cite[Corollary 8.45]{Watson}) \begin{equation}\label{v-limit}\lim_{(X,t)\rightarrow (X_0,t_0)}v_j(X,t)=\overline{u}_j(X_0,t_0)\end{equation} at nearly every point of continuity of $\overline{u}_j|_{\partial_e\Omega_1}$ by (i), with the usual interpretation of the limit at semi-singular boundary points. Thus, as we already noted that $\overline{u}_j$ is nearly everywhere continuous, \eqref{v-limit} holds at nearly every $(X_0,t_0)\in\partial_e\Omega_1$.

Now, the functions $u_j$, $\overline{u}_j$, and $v_j$ are bounded, taking values between $0$ and $1$ at all points where they are defined, because $0\leq \phi_j\leq 1$ and caloric measures are probabilities. Hence $\pm(u_j|_{\Omega_1}-v_j)$ are (hypo) temperatures in $\Omega_1$ satisfying \eqref{max-1}. At nearly every point $(X_0,t_0)\in\partial_e\Omega_1$, with the usual interpretation when $(X_0,t_0)\in\partial_{ss}\Omega_1$, we also have $$\lim_{(X,t)\rightarrow (X_0,t_0)} \pm(u_j|_{\Omega_1}(X,t)-v_j(X,t)) =\pm(\overline{u}_j(X_0,t_0)-\overline{u}_j(X_0,t_0))=0,$$ where we used $u_j(X,t)=\overline{u}_j(X,t)$ for all $(X,t)\in\Omega_1$, because $\Omega_1\subset\Omega_2$. By the maximum principle (see Theorem \ref{maximum principle}), we conclude that $u_j(X,t)=v_j(X,t)$ at every $(X,t)\in\Omega_1$.
Therefore, for every $(X,t)\in\Omega_1$, \begin{equation}\label{abc1} \int_{\partial_e\Omega_2}\phi_j\,d\omega^{X,t}_{\Omega_2}
= \int_{A_1\cup B_1\cup C_1} \overline{u}_j\,d\omega^{X,t}_{\Omega_1}=\int_{A_1} u_j\,d\omega^{X,t}_{\Omega_1}+\int_{B_1} \phi_j\,d\omega^{X,t}_{\Omega_1}.\end{equation} Note that $\int_{C_1} \overline{u}_j\,d\omega^{X,t}_{\Omega_1}$ vanished, because $C_1$ is polar.

To conclude, evaluate \eqref{abc1} at a fixed point $(X_1,t_1)\in\Omega_1$ and send $j\rightarrow\infty$. Recall that $\phi_j\downarrow \chi_E$ pointwise and $0\leq \phi_j\leq 1$. By either the monotone or dominated convergence theorem, the left hand side of \eqref{abc1} becomes $$\int_{\partial_e \Omega_2} \chi_E\,d\omega^{X_1,t_1}_{\Omega_2}=\omega^{X_1,t_1}_{\Omega_2}(E).$$ Similarly, taking into account the fact that caloric measures on $\Omega_1$ are supported in $\partial_e\Omega_1$ and $E\subset\partial_e\Omega_2$, the second term on the right hand side of \eqref{abc1} becomes $\omega^{X_1,t_1}_{\Omega_1}(E)$. Finally, send $j\rightarrow\infty$ in the definition of $u_j(Z,\tau)$ to get $\lim_{j\rightarrow\infty} u_j(Z,\tau)=\omega^{Z,\tau}_{\Omega_2}(E)$ pointwise. By one more application of the dominated convergence theorem, the first term on the right hand side of \eqref{abc1} transforms into $$\int_{A_1} \omega^{Z,\tau}_{\Omega_2}(E)\,d\omega^{X_1,t_1}_{\Omega_1}.$$ This verifies \eqref{StrongMarkov} for compact sets $E\subset\partial_e\Omega_2$. \end{proof}

Extending the notion of a parabolic cube, a \emph{(closed) parabolic rectangle} in $\RR^{n+1}$ is any set $H$ of the form $H=[X_1,X_1+s_1]\times\cdots[X_n,X_n+s_n]\times[X_{n+1},X_{n+1}+s_{n+1}^2]$; we call the numbers $s_1,\dots,s_n,s_{n+1}>0$ the \emph{side lengths} of $H$. Iterating the strong Markov property, we obtain the following estimate, which we will use in the proof of Lemma \ref{lemma2}.

\begin{corollary}[caloric measure of nested rectangles] \label{c-nested} Let $\Omega\subsetneq\RR^{n+1}$ be a domain and suppose nearly every point of $\partial_e\Omega$ is regular for the Dirichlet problem for the heat equation. Let $H_1,\dots,H_k$ be parabolic rectangles in $\RR^{n+1}$ that are strictly nested in the sense that $$H_k\subset\interior{H_{k-1}},\quad H_{k-1}\subset\interior H_{k-2},\quad \dots,\quad  H_2\subset\interior{H_1}.$$ Write $G_i'=\Omega\cap\partial_e(\Omega\setminus H_i)\subset \partial H_i=G_i$ for each $i$. If $(X,t)\in \Omega\setminus H_1$, then $$\omega^{X,t}_\Omega(H_k)\leq \omega^{X,t}_{\Omega\setminus H_1}(G_1')\left(\sup_{(X_1,t_1)\in G_1'} \omega^{X_1,t_1}_{\Omega\setminus H_2}(G_2')\right)
\cdots\left(\sup_{(X_{k-1},t_{k-1})\in G_{k-1}'}\omega^{X_{k-1},t_{k-1}}_{\Omega\setminus H_{k-1}}(G_k)\right).$$ (Except for the final instance $G_k$, all instances of a `\,$G$' in the formula are $G_i'$.)\end{corollary}

\begin{proof} Let us mention once and for all that removing a rectangle $H$ from a domain $\Omega$ enlarges the complement of the domain. In particular, $\partial_{ss}(\Omega\setminus H)\subset \partial_{ss}\Omega$ and every regular point of $\partial_e\Omega$ that belongs to $\partial_e(\Omega\setminus H)$ is also a regular point of $\partial_e(\Omega\setminus H)$. Moreover, every point in $\partial_e(\Omega\setminus H)\cap\partial H$ is regular; e.g.~use \cite[Theorem 8.49]{Watson} to check regularity of any points in $\partial_e(\Omega\setminus H)\cap \partial H$ that do not belong to $\partial_e\Omega$.

We use induction on the number of rectangles. For the base case, suppose that $k=1$. Suppose that nearly every point in $\partial_e\Omega$ is regular for the Dirichlet problem on $\Omega$. Hence nearly every point in $\partial_e(\Omega\setminus H_1)$ is regular for the Dirichlet problem on $\Omega\setminus H_1$, as well. Comparing boundary values, one may check that $\lim_{(X,t)\rightarrow (Y,s)}(\omega^{X,t}_\Omega(H_k)-\omega^{X,t}_{\Omega\setminus H_1}(G_1))\leq 0$ at nearly every $(Y,s)\in\partial_e(\Omega\setminus H_1)$, with the standard proviso that $t>s$ when $(Y,s)\in\partial_{ss}\Omega$. Thus, by the maximum principle, $\omega^{X,t}_{\Omega}(H_k) \leq \omega^{X,t}_{\Omega\setminus H_1}(G_1)$ for all $(X,t)\in \Omega\setminus H_1$.

Suppose that the corollary holds for some $k\geq 1$. Let $H_1,\cdots, H_{k+1}$ be rectangles with $H_{j+1}\subset \interior{H_{j}}$ for all $1\leq j\leq k$ and fix $(X,t)\in\Omega\setminus H_1$. Note that every $(X_1,t_1)\in G_1'$ lies outside of $H_2$. Thus, the inductive hypothesis applied with $H_2,\cdots, H_{k+1}$ guarantees that $$\omega^{X_1,t_1}_\Omega(H_{k+1})\leq \omega^{X_1,t_1}_{\Omega\setminus H_2}(G_2')\left(\sup_{(X_2,t_2)\in G_2'} \omega^{X_2,t_2}_{\Omega\setminus H_3}(G_3')\right)
\cdots\left(\sup_{(X_{k},t_{k})\in G_{k}'}\omega^{X_{k},t_{k}}_{\Omega\setminus H_{k+1}}(G_{k+1})\right).$$ (When $k=1$, this formula should be read as $\omega^{X_1,t_1}_\Omega(H_2) \leq \omega^{X,t}_{\Omega\setminus H_2}(G_2).$) Since $\Omega\setminus H_1\subset \Omega$, the strong Markov property ensures that \begin{align*}
\omega^{X,t}_\Omega(H_{k+1}) &= \omega^{X,t}_{\Omega\setminus H_1}(H_{k+1}\cap\partial_e\Omega)
   + \int_{\Omega\cap \partial_e(\Omega\setminus H_1)} \omega^{X_1,t_1}_{\Omega}(H_{k+1})\,d\omega^{X,t}_{\Omega\setminus H_1}(X_1,t_1)\\
&= \int_{G_1'} \omega^{X_1,t_1}_{\Omega}(H_{k+1})\,d\omega^{X,t}_{\Omega\setminus H_1}(X_1,t_1)
\leq \omega^{X,t}_{\Omega\setminus H_1}(G_1')\sup_{(X_1,t_1)\in G_1'} \omega^{X_1,t_1}_{\Omega}(H_{k+1}),\end{align*} where $\omega^{X,t}_{\Omega\setminus H_1}(H_{k+1}\cap\partial_e\Omega)=0$ trivially, since $H_{k+1}$ is contained in the exterior of $\Omega\setminus H_1$. Combining the two displayed equations yields the desired inequality for $H_1,\dots,H_{k+1}$. This completes the induction step.\end{proof}

\section{Lower dimension bound} \label{sec:lower-bound}

The \emph{fundamental temperature} (also known as the \emph{heat kernel}) on $\RR^{n+1}$ is the function $W(X,t):\RR^n\times\RR\rightarrow[0,\infty)$ defined by $$W(X,t)=(4\pi t)^{-n/2}\exp(-|X|^2/4t)\chi_{\{t>0\}}.$$ It is a solution to the heat equation on $\RR^{n+1}\setminus\{0\}$, is a supersolution of the heat equation on $\RR^{n+1}$, and satisfies $\int_{\RR^n} W(X-Y,t-s)\,dY=1$ for all $X\in\RR^n$ and $t,s\in\RR$ with $t>s$. To make effective estimates of caloric measure on arbitrary domains in this section and the next, we need to understand the phase portrait for the fundamental temperature.

\begin{lemma}[portrait of $W(|X|,t)$]\label{fundprop} Let $\phi:[0,\infty)\times(0,\infty)\rightarrow[0,\infty)$ be defined by $$\phi(r,t)=(4\pi t)^{-n/2}\exp(-r^2/4t).$$ \begin{itemize}
\item \emph{horizontal traces:} For all $t>0$, the function $r\mapsto \phi(r,t)$ is strictly decreasing, $\phi(0,t)=(4\pi t)^{-n/2}$, and $\lim_{r\rightarrow \infty} \phi(r,t)=0$.
\item \emph{vertical traces:} The function $t\mapsto \phi(0,t)=(4\pi t)^{-n/2}$ is strictly decreasing, $\lim_{t\rightarrow 0+}\phi(0,t)=\infty$, and $\lim_{t\rightarrow\infty} \phi(0,t)=0$.
\item For all $r>0$, the function $t\mapsto \phi(r,t)$ increases on $(0,(1/2n)r^2]$, decreases on $[(1/2n)r^2,\infty)$, $\lim_{t\rightarrow 0+}\phi(r,t)=\lim_{t\rightarrow\infty}\phi(r,t)=0$, and has maximum value $$\max_{t>0} \phi(r,t) = \phi(r,(1/2n)r^2)=(n/2\pi)^{n/2}\exp(-n/2)r^{-n}=C_{n}r^{-n},$$ where $C_n\rightarrow\infty$ superexponentially as $n\rightarrow\infty$, although $C_n>C_{n+1}$ for $1\leq n\leq 5$.
\item \emph{parabolic traces:} For all $\lambda>0$, $\phi(r,\lambda r^2)=C(\lambda,n)r^{-n}$, where as a function of $\lambda$ the constant $C(\lambda,n)$ increases on $(0,1/2n]$ and decreases on $[1/2n,\infty)$.
\item \emph{parabolic dilation:} For all $\lambda>0$, $\phi(\lambda r,\lambda^2t)=\lambda^{-n}\phi(r,t)$.
\end{itemize}
\end{lemma}

\begin{remark} The proof of Lemma \ref{fundprop} is an easy exercise in calculus. With the aid of a computer, it can be checked that initially the constants $C_n=\sup_{\lambda>0} C(\lambda,n)$ satisfy $$.25>C_1>C_2>C_3>C_4>C_5>C_6>.04,$$ but thereafter $C_6<C_7<C_8<\cdots$.\end{remark}

On $\RR^{n+1}$, let $U_{r,s}(X,t)=U_{\RR^n}(X,r)\times(t-s,t+s)$ denote the \emph{open cylinder} with \emph{center} $(X,t)$, \emph{radius} $r$, and \emph{duration} $2s$. A special case is $U_{r,r^2}(X,t)=U((X,t),r)$, an open ball of radius $r$ in $\RR^{n+1}$ with the parabolic distance. We refer to \eqref{e:cylinder-1} as a universal estimate, because it holds on every domain $\Omega\subsetneq\RR^{n+1}$.

\begin{lemma}[universal cylinder estimate] \label{cylinder-estimate} Let $\Omega\subsetneq\RR^{n+1}$ be a domain and let $(X,t)\in\Omega$. For all $(X_0,t_0)\in\partial_e\Omega$, $r>0$, and $s>0$, \begin{equation}\label{e:cylinder-1} \omega_\Omega^{X,t}(U_{r,s}(X_0,t_0))\leq \frac{W(X-X_0,(1/2n)r^2+s+t-t_0)}{\phi(1,(1/2n)+2s/r^2)}\,r^n.\end{equation}\end{lemma}

\begin{proof} Write $E=U_{r,s}(X_0,t_0)$. The following argument may be viewed as a modification of the usual proof that caloric measure of the present and future is zero (e.g.~see \cite[Lemma 2.10]{Watson} or Example (a) on \cite[p.~329]{Doob}). The basic idea is to show that an appropriately chosen translation and dilation of the fundamental temperature belongs to the upper class $\mathfrak{U}^{f}_{\Omega}$ for $f=\chi_E$; see \eqref{e:upper}. Define $$v(X,t)=\frac{W(X-X_0,t-t_{-1})}{R}\quad\text{for all }(X,t)\in\RR^{n+1}$$ for some time $t_{-1}<t_0-s$ and number $R>0$ to be determined. As long as we choose \begin{equation}\label{R-choice1} 0<R\leq \inf_{(X,t)\in E} W(X-X_0,t-t_{-1}) =\inf_{\substack{0\leq \rho<r \\ t_0-s<\tau<t_0+s}}\phi(\rho,\tau-t_{-1})=\inf_{t_0-s< \tau< t_0+s}\phi(r,\tau-t_{-1}),\end{equation} where $\phi$ is from Lemma \ref{fundprop}, we have $$\liminf_{(X,t)\rightarrow (X_1,t_1)} v(X,t) \geq 1=\chi_E(X_1,t_1)\quad\text{for every }(X_1,t_1)\in(\partial_e\Omega)\cap E.$$ (For equality of the infima in \eqref{R-choice1}, we used $\rho\mapsto \phi(\rho,\tau)$ is decreasing for each $\tau>0$.) Also, we always have $$\liminf_{(X,t)\rightarrow (X_1,t_1)} v(X,t) \geq 0=\chi_E(X_1,t_1)\quad\text{for every }(X_1,t_1)\in(\partial_e\Omega)\setminus E,$$ simply because $v$ is nonnegative. (Luckily, because $v$ is defined on all of $\RR^{n+1}$, we do not need any special interpretation of the limit inferior when $(X_1,t_1)\in\partial_{ss}\Omega$.) Hence $v\in \mathfrak{U}^f_\Omega$. It immediately follows that $\omega^{X,t}_\Omega(E)=H^f_\Omega(X,t)\leq v(X,t)$ for all $(X,t)\in\Omega$.

Clearly, the best upper bound from this comparison is found by taking $$R=\min_{t_0-s\leq \tau\leq t_0+s} \phi(r,\tau-t_{-1})$$ and choosing $t_{-1}$ to make $R$ as large as possible. Rather than try to find the optimal value of $t_{-1}$, which would depend on $n$, $r$, and $s$ and seems hard, we use our knowledge of the vertical traces of $\phi$ in Lemma \ref{fundprop} to make a smart choice. Set $t_{-1}=t_0-s-(1/2n)r^2$, which puts $\phi(r,t_0-s-t_{-1})=\phi(r,(1/2n)r^2)=\max_{\tau>0}\phi(r,\tau)=C_nr^{-n}$ as large as possible. Since $\phi(r,\tau)$ decreases on $[(1/2n)r^2,\infty)$, it follows that with our value of $t_{-1}$, $$R=\phi(r,(1/2n)r^2+2s)=r^{-n} \phi(1,(1/2n)+2s/r^2).$$ Putting it all together, $$\omega^{X,t}_\Omega(U_{r,s}(X_0,t_0)) \leq \frac{W(X-X_0,t-t_{-1})}{R}=\frac{W(X-X_0, (1/2n)r^2+s+t-t_0)}{\phi(1,(1/2n)+2s/r^2)}r^n$$ for all $(X,t)\in\Omega$. We have arrived at \eqref{e:cylinder-1}.\end{proof}

\begin{lemma}[universal ball estimate] \label{ball-estimate} Let $\Omega\subsetneq\RR^{n+1}$ be any domain and let $(X,t)\in\Omega$. For all $(X_0,t_0)\in\partial_e\Omega$ and $0<r\leq \dist((X,t),\partial_e\Omega)$, \begin{equation}\label{e:cylinder-3} \omega_\Omega^{X,t}(U((X_0,t_0),r)) \lesssim_n \frac{r^n}{\dist((X,t),\partial_e\Omega)^n}.\end{equation}\end{lemma}

\begin{proof} Let $(X,t)\in\Omega$, $(X_0,t_0)\in\partial_e\Omega$, and $r>0$. By Lemma \ref{cylinder-estimate}, $$\omega^{X,t}_\Omega(U((X_0,t_0),r))=\omega^{X,t}_\Omega(U_{r,r^2}(X_0,t_0))\lesssim_n W(X-X_0,(1+1/2n)r^2+t-t_0)r^n,$$ where the constant $\phi(1,1/2n+2)^{-1}\leq \phi(1,5/2)^{-1}\leq e^{1/10}(10\pi)^{n/2}$. To continue, write $$\delta=\dist((X,t),\partial_e\Omega)\leq \max(|X-X_0|,|t-t_0|^{1/2}).$$ Suppose that $|X-X_0|\geq \delta$, i.e.~the space distance is large. By Lemma \ref{fundprop}, $$W(X-X_0,
(1+1/2n)r^2+t-t_0)\leq \phi(\delta,(1+1/2n)r^2+t-t_0)\leq \phi(\delta,(1/2n)\delta^2)=C_n\delta^{-n}.$$ This yields \eqref{e:cylinder-3} when $|X-X_0|\geq \delta$. Of course,
the other possibility is that $|X-X_0|<\delta$, but $|t-t_0|\geq \delta^2$, i.e.~the space distance is small, but the time distance is large. If $t\leq t_0-r^2$, then $\omega^{X,t}(U((X_0,t_0),r))=0$ by Remark \ref{zero-future} and the estimate \eqref{e:cylinder-3} is trivially true. Thus, we may suppose that $t>t_0-r^2$. We now impose the hypothesis $r\leq \delta$, which ensures that $t>t_0-\delta^2$. Since $|X-X_0|<\delta$ and $|t-t_0|\geq \delta^2$, it follows that in fact $t\geq t_0+\delta^2$. By Lemma \ref{fundprop}, $$W(X-X_0,(1+1/2n)r^2+t-t_0)\leq  \phi(0,(1+1/2n)r^2+t-t_0) \leq \phi(0,\delta^2)=(4\pi)^{-n/2}\delta^{-n}.$$ This yields \eqref{e:cylinder-3} when $|X-X_0|<\delta$ and $|t-t_0|\geq \delta$ (and $r\leq\delta$).\end{proof}

\begin{remark}The proof of \eqref{e:cylinder-3} that we gave produces an implicit constant, which grows superexponentially as $n\rightarrow\infty$. It would be interesting to find the best possible constant and/or find geometric conditions (flatness?) under which the exponent $n$ can be improved.\end{remark}

We are ready to prove the first half of the main theorem.

\begin{proof}[Proof of {Theorem \ref{main}(i)}] Fix any domain $\Omega\subsetneq\RR^{n+1}$, pole $(X,t)\in\Omega$, and caloric measure $\omega=\omega^{X,t}_\Omega$. We wish to prove $\omega\ll\Haus^n$. Suppose that $E\subset\RR^{n+1}$ is any set with $\Haus^n(E)=0$. Let $G=E\cap \partial_e\Omega\cap \big(\RR^n\times(-\infty,t)\big)$ and let $F=E\setminus G$. On the one hand, $\omega(F)=0$ by \eqref{omega-support}. On the other hand, our supposition guarantees that $\Haus^n(G)\leq \Haus^n(E)=0$. Fix $\delta>0$. By Lemma \ref{spherical-null-sets}, we can find a sequence $U_i=U((X_i,t_i),r_i)$ of open balls centered at $(X_i,t_i)\in G$ of radius $0<r_i<\delta$ such that $G\subset \bigcup_1^\infty U_i$, and $\sum_1^\infty (2r_i)^n<\delta$. As long as $\delta$ is sufficiently small (in particular, we require $\delta\leq \dist((X,t),\partial_e\Omega)$), Lemma \ref{ball-estimate} gives $$\omega(G) \leq \sum_1^\infty \omega(U((X_i,t_i),r_i)) \lesssim_n
\sum_1^\infty \frac{r_i^n}{\dist((X,t),\partial_e\Omega)^n} \lesssim_n \frac{\delta}{\dist((X,t),\partial_e\Omega)^n}.$$ Sending $\delta\rightarrow 0$, we conclude that $\omega(G)=0$. Thus, $\omega(E)\leq \omega(F)+\omega(G)=0$. This verifies $\omega$ is absolutely continuous with respect to $\Haus^n$. It follows that if $E\subset\RR^{n+1}$ is a Borel set with $\omega(E)>0$, then $\Haus^n(E)>0$ and $\dim_H E\geq n$. Therefore, $\underline{\dim}_H\,\omega \geq n$.\end{proof}

\section{Bourgain's alternative and upper dimension bound} \label{sec:upper-bound}

Write $\gap(A,B)=\inf\{\dist((X,t),(Y,s)):(X,t)\in A,\,(Y,s)\in B\}$ for the \emph{gap} between sets in $A,B\subset\RR^{n+1}$ with the parabolic distance. (This terminology comes from variational analysis. Harmonic analysts may be more familiar with the notation $\dist(A,B)$. As gap does not satisfy the triangle inequality, the latter notation is perhaps unwise.)

The following lemma is the analogue of \cite[Lemma 1]{Bourgain} in the parabolic context. See Figure \ref{fig:alt}.

\begin{lemma}[Bourgain's alternative for caloric measure] \label{alternative} Fix $\Delta=\Delta^m(\RR^{n+1})$ for some $n\geq 1$ and $m\geq 2$; see \S\ref{ss:cubes}. Let $Q$ be an open parabolic cube in $\RR^{n+1}$ of side length $r$, i.e. $$Q=z_Q+(-r/2,r/2)^n\times (-r^2,0)\quad\text{for some point }z_Q=(X_Q,t_Q)\in\RR^{n+1}.$$ Let $Q_*=z_Q+S\times [-3\epsilon^2r^2,-2\epsilon^2r^2]$ be a parabolic cube in $\RR^{n+1}$ of side length $\epsilon r$ contained in $Q$ such that $\gap(Q_*,\partial_n Q)\geq \delta r$. Let $F=z_Q+S\times [-\epsilon^2r^2,0)$. Suppose that $\epsilon/\delta\leq 1/\sqrt{6n}$. For all closed sets $E\subset\RR^{n+1}$ and $0<\eta<1$, at least one (possibly both) of the following alternatives hold: \begin{align}\label{alternative-1} \omega^{X,t}_{Q\setminus E}(E\cap Q)\geq \eta\!\quad\quad\quad&\text{for all }(X,t)\in F\setminus E,\\
\label{alternative-2} \Net^\rho_\infty(E\cap Q_*)\lesssim_{n}\frac{m^\rho}{1-m^{n-\rho}}\,\eta(\epsilon r)^\rho\quad&\text{for all }n< \rho\leq n+2,\end{align} where $\Net^\rho_\infty$ denotes the $\rho$-dimensional $m$-adic parabolic net content relative to $\Delta$.\end{lemma}

\begin{remark}\label{n-only} (i) By Remark \ref{hausdorff-net}, one can replace $\Net^\rho_\infty$ in \eqref{alternative-2} with the parabolic Hausdorff content $\Haus^\rho_\infty$. (It is then best to pick $m=2$.)

(ii) We observe that $(1-m^{n-\rho})^{-1}\leq (1-2^{n-\rho})^{-1}\leq 2$ when $n+1\leq \rho\leq n+2$.\end{remark}

\begin{figure}\begin{center}\includegraphics[width=.8\textwidth]{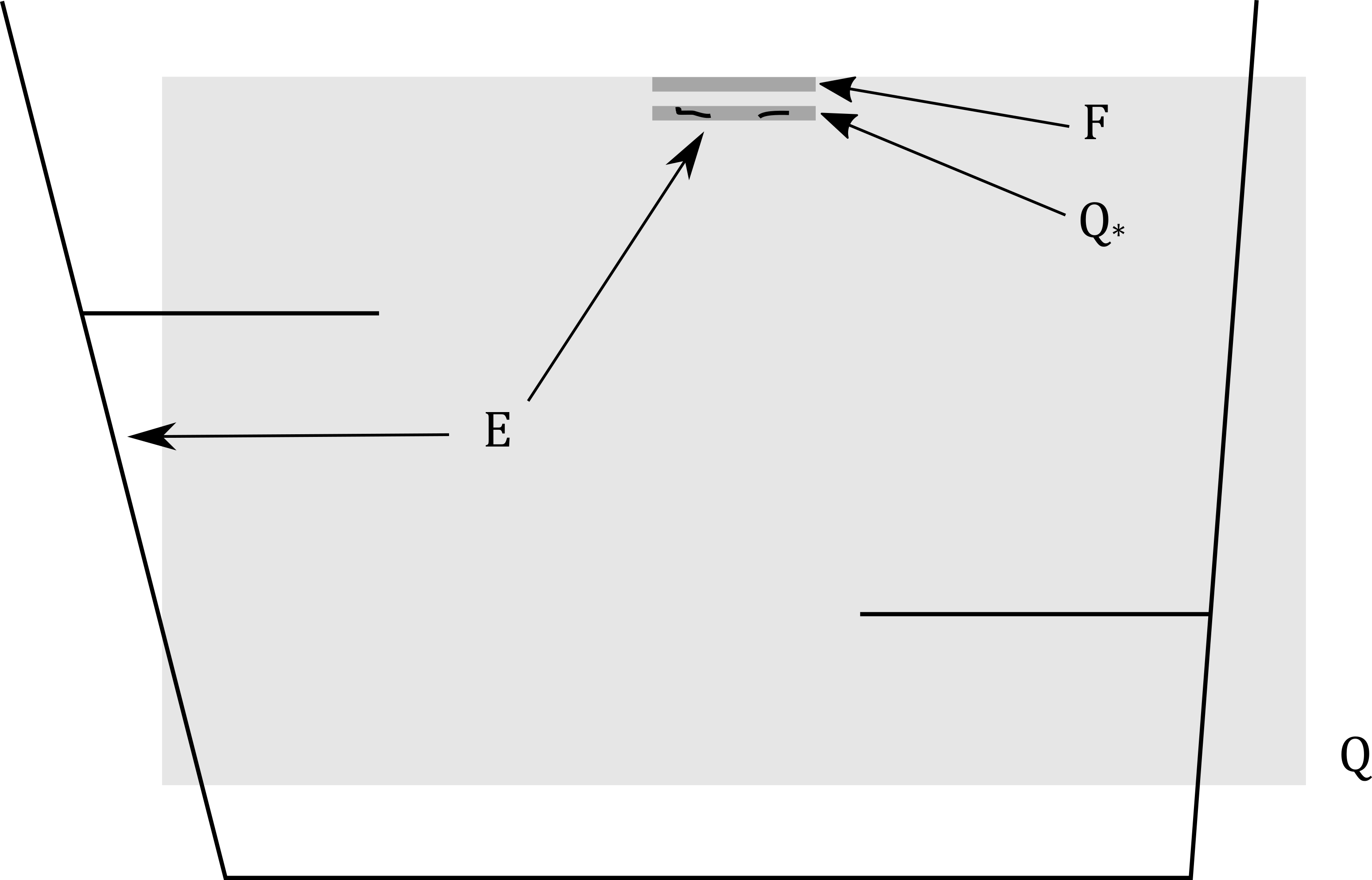}\end{center}\caption{Bourgain's alternative for a cube $Q$ in $\RR^2$ with size parameters $r/r^2=(1+\sqrt{5})/2$, $\epsilon=1/7$, and $\delta=3/7$. If the content of $E\cap Q_*$ is large, then caloric measure $\omega^{X,t}_{Q\setminus E}(E\cap Q)$ is large for every pole $(X,t)\in F$.} \label{fig:alt} \end{figure}

\begin{proof}[Proof of Lemma {\ref{alternative}}] Fix $\Delta=\Delta^m(\RR^{n+1})$. Let $Q$ be an open parabolic cube in $\RR^{n+1}$ with side length $r$ and center $z_Q$. Choose associated subcubes $Q_*$ and $F$ with common face $S$ and parameters $\delta$ and $\epsilon$ as stated. The center and side length of the cube $Q$ play no important role in the proof that follows, but the astute reader might rightly object to normalizing $z_Q=0$ and $r=1$, as the net contents $\Net^\rho_\infty$ are neither translation nor dilation invariant. Because we do not wish to import any additional dependence on $m$ in \eqref{alternative-2} beyond what is stated, we shall work in the general setup.

Let $E\subset\RR^{n+1}$ be closed and let $0<\eta<1$. Suppose that \eqref{alternative-1} fails. Then we must establish \eqref{alternative-2}. Freeze $n<\rho\leq n+2$. If $\Haus^\rho(E\cap Q_*)=0$, then $\Net^\rho_\infty(E\cap Q_*)=0$ and \eqref{alternative-2} is immediate for the fixed value of $\rho$. Thus, we may assume that $\Haus^\rho(E\cap Q_*)>0$. By Theorem \ref{frostman}, there exists a compact set $K\subset E\cap Q_*$ and finite Borel measure $\mu$ on $\RR^{n+1}$ with support in $K$ such that \begin{equation}\label{mu-upper-ar}
\mu(R)\leq (\side R)^\rho\quad\text{for all }R\in\Delta\text{ and}\end{equation}
\begin{equation}\label{mu-lower-mass}\mu(K)\gtrsim_n \Net^\rho_\infty(E\cap Q_*).\end{equation}
Consider the potential $u(X,t)=\int_K W(X-Y,t-s)\,d\mu(Y,s)$ for $(X,t)\in\RR^{n+1}\setminus K$, which solves the heat equation by differentiation under the integral. To proceed, let us momentarily take for granted three estimates (proved below):
\begin{equation*} \tag{E1} \label{estimate1}
    u(X,t)\lesssim_{n} \frac{m^\rho}{1-m^{n-\rho}}(\epsilon r)^{\rho-n} \quad\text{for all }(X,t)\in \RR^{n+1}\setminus K,
\end{equation*}
\begin{equation*}\tag{E2} \label{estimate2}
    u(X,t)\geq (12\pi)^{-n/2}e^{-n/12}(\epsilon r)^{-n}\mu(K)\quad\text{for all }(X,t)\in F,
\end{equation*}
\begin{equation*}\tag{E3} \label{estimate3}
    u(X,t)\leq(12\pi)^{-n/2}e^{-n/2}(\epsilon r)^{-n}\mu(K)\quad\text{for all }(X,t)\in\partial_n Q,
\end{equation*} where (E1) requires $\rho>n$ and (E3) holds provided that $\epsilon/\delta \leq 1/\sqrt{6n}$. Define an auxiliary temperature $w$ on $\RR^{n+1}\setminus K$ by setting $$w(X,t)=\frac{u(X,t)-\|u\|_{L^\infty(\partial_n Q)}}{\|u\|_{L^\infty(Q\setminus K)}}\quad\text{for all }(X,t)\in\RR^{n+1}\setminus K.$$
By design, $w\leq 0$ on $\partial_n Q$, $w$ is continuous on $\partial_n Q$, and $w\leq 1$ on all of $Q\setminus K$. Thus, $$\limsup_{(X,t)\rightarrow (X_0,t_0)} w(X,t) \leq \chi_K(X_0,t_0)\quad\text{for all }(X_0,t_0)\in\partial(Q\setminus K)\setminus\partial_s Q.$$ Note that $\partial(Q\setminus K)\setminus\partial_s Q$ contains $\partial_e(Q\setminus K)$. Hence $w$ belongs to the lower class  $\mathfrak{L}^f_{Q\setminus K}$ for $f=\chi_K$; see \eqref{e:lower}. Thus, $w(X,t)\leq H^f_{Q\setminus K}(X,t)=\omega^{X,t}_{Q\setminus K}(K)$ for every pole $(X,t)\in Q\setminus K$.

\emph{Brief interlude.} We claim that $\omega^{X,t}_{Q\setminus K}(K)\leq \omega^{X,t}_{Q\setminus E}(E\cap Q)$ for every pole $(X,t)\in Q\setminus E$. As it is possible that $\partial_e(Q\setminus E)\cap\partial_s Q\neq\emptyset$ and we do not have good control on the boundary behavior of $\omega^{X,t}_{Q\setminus K}(K)$ as $(X,t)$ approaches $\partial_s Q$, this requires a short argument. To check that the comparison holds at some fixed pole $(X_1,t_1)\in Q\setminus E$, put $$E'=E\cap\big(\RR^n\times(-\infty,t_1]\big)$$ and note that $\omega^{X_1,t_1}_{Q\setminus E}(E\cap Q)=\omega^{X_1,t_1}_{Q\setminus E'}(E'\cap Q)$ by Remark \ref{zero-future}. That is to say, we can modify the part of the domain forward-in-time from $(X_1,t_1)$ without changing the value of the caloric measure at $(X_1,t_1)$. Because $E'$ and $\partial_s Q$ are separated in time, we have $\partial_s(Q\setminus E')\supset \partial_s Q$. Hence $\partial_e(Q\setminus E') \subset \partial_n Q \cup (E'\cap Q)$. Now, $\lim_{(X,t)\rightarrow (X_0,t_0)} \omega^{X,t}_{Q\setminus K}(K)=0$ for all $(X_0,t_0)\in\partial_n Q$ and $\limsup_{(X,t)\rightarrow (X_0,t_0)}\omega^{X,t}_{Q\setminus K}(K) \leq 1$ for all $(X_0,t_0)\in E'\cap Q$. Thus, $\omega^{X,t}_{Q\setminus K}(K)$ belongs to the lower class $\mathfrak{L}^g_{Q\setminus E'}$ for $g=\chi_{E'\cap Q}$. Therefore, $\omega^{X_1,t_1}_{Q\setminus K}(K)\leq H^g_{Q\setminus E'}(X_1,t_1)=\omega^{X_1,t_1}_{Q\setminus E'}(E'\cap Q)=\omega^{X_1,t_1}_{Q\setminus E}(E\cap Q)$, as required.

To continue, suppose that $(X,t)\in F\setminus E$. Using (E1) to estimate the denominator in the definition of $w$ and (E2), (E3) to estimate the numerator, we see that
\begin{align*}
    w(X,t)&\gtrsim_{n} (1-m^{n-\rho})\frac{(12\pi)^{-n/2}\big(e^{-n/12}-e^{-n/2}\big)(\epsilon r)^{-n}\mu(K)}{m^{\rho}(\epsilon r)^{\rho-n}}\gtrsim_n (1-m^{n-\rho})\frac{\mu(K)}{m^\rho (\epsilon r)^\rho}.
\end{align*} Recalling \eqref{mu-lower-mass} and $\omega^{X,t}_{Q\setminus E}(E\cap Q)\geq \omega^{X,t}_{Q\setminus K}(K)\geq w(X,t)$, it follows that \begin{equation}\label{b-estimate} \frac{m^\rho(\epsilon r)^\rho}{1-m^{n-\rho}}\,\omega^{X,t}_{Q\setminus E}(E\cap Q) \gtrsim_n \Net^\rho_\infty(E\cap Q_*).\end{equation} By supposition, \eqref{alternative-1} fails, i.e.~$\omega^{X,t}_{Q\setminus E}(E\cap Q)<\eta$ for some $(X,t)\in F\setminus E$. Together with \eqref{b-estimate}, this yields \eqref{alternative-2} for the fixed value of $\rho$. Since we fixed $\rho$ arbitrarily, this completes the proof of the lemma, assuming \eqref{estimate1}, \eqref{estimate2}, and \eqref{estimate3}.
\end{proof}

\begin{proof}[Proof of {\eqref{estimate1}}]
For the duration of the proof, write \begin{equation}\label{dist-infty}\dist_\infty((X,t),(Y,s))=\max\left(2\|X-Y\|_\infty,\sqrt{2}|t-s|^{1/2}\right),\end{equation} where $\|X-Y\|_\infty=\max_{i=1}^n|X_i-Y_i|$ is the $L^\infty$ distance on $\RR^n$. The metric $\dist_\infty$ has the feature that its closed balls $B_\infty((X,t),\lambda)=\{(Y,s):\dist_\infty((X,t),(Y,s))\leq \lambda\}$ are closed parabolic cubes in $\RR^{n+1}$ of side length $\lambda$ with center at $(X,t)$.

Fix $(X,t)\in \RR^{n+1}\setminus K$ and let $j_0$ denote the integer such that $m^{-(j_0+1)} < \side Q_* \leq m^{-j_0}.$
We start by estimating the $\mu$ measure of certain sets that are needed below. Recall that $\mu(R)\leq (\side R)^\rho$ for all $R\in\Delta$ by \eqref{mu-upper-ar}. Since $K\subset Q_*$, we can cover $K$ by $3^{n+1}$ or fewer cubes $R\in \Delta$ of side length $m^{-j_0}$ such that $\overline{R}\cap\overline{Q_*}\neq\emptyset$. Hence $$\mu(K)\leq 3^{n+1} m^{-\rho j_0}.$$ Similarly, $\mu(B_\infty((X,t),m^{-j}))\leq 3^{n+1} m^{-\rho j}$ for all $j\in\ZZ$. To upper bound the potential $u(X,t)=\int_K W(X-Y,t-s)\,d\mu(Y,s)$, we now split into two cases.

\emph{Case 1.} Suppose that $(X,t)$ is far away from $K$ in the sense that $\dist_\infty((X,t),K)=\inf_{(Y,s)\in K}\dist_\infty((X,t),(Y,s))>m^{-j_0}$. Then for any $(Y,s)\in K$, we have $$|X-Y|\geq \|X-Y\|_\infty>\frac12 m^{-j_0}\quad\text{or}\quad|t-s|> \frac12 m^{-2j_0}.$$ In the first scenario, when $t>s$, $W(X-Y,t-s) \leq \phi(\tfrac{1}{2}m^{-j_0},t-s)\lesssim_n m^{nj_0}$ by Lemma \ref{fundprop} (the maximum bound on vertical traces). In the second scenario, when $t>s$, $$W(X-Y,t-s)\leq \phi(0,t-s) \leq \phi(0,\tfrac12 m^{-2j_0}) \lesssim_n m^{nj_0},$$ as well. And, of course, $W(X-Y,t-s)=0$ whenever $t\leq s$. All together, we obtain $u(X,t)=\int_K W(X-Y,t-s)\,d\mu(Y,s) \lesssim_n m^{nj_0}\mu(K) \lesssim_n m^{-j_0(\rho-n)}
\lesssim_n m^{\rho-n} (\epsilon r)^{\rho-n}$, where the last inequality follows from the choice of $j_0$.

\emph{Case 2.} Suppose that $(X,t)$ is nearby $K$, meaning $\dist_\infty((X,t),K)\leq m^{-j_0}$. For each $j\geq j_0$, define the annulus $A_j=B_\infty((X,t),m^{-j})\setminus B_\infty((X,t),m^{-(j+1)})$. \emph{Mutatis mutandis}, the argument in Case 1 yields $W(X-Y,t-s) \lesssim_n m^{n(j+1)}$ for all $(Y,s)\in K\cap A_j$. Hence \begin{equation*}\begin{split}
u(X,t)&= \sum_{j=j_0}^\infty \int_{K\cap A_j} W(X-Y,t-s)\,d\mu(Y,s)\lesssim_n \sum_{j=j_0}^\infty m^{n(j+1)}\mu(A_j)\\
&\lesssim_n m^n\sum_{j=j_0}^\infty m^{-j(\rho-n)}\lesssim_n m^n\frac{m^{-j_0(\rho-n)}}{1-m^{n-\rho}}\lesssim_n m^\rho \frac{(\epsilon r)^{\rho-n}}{1-m^{n-\rho}}.\end{split} \end{equation*} This verifies \eqref{estimate1}.
\end{proof}

\begin{proof}[Proof of {\eqref{estimate2}}]
Let $(X,t)\in F$ and $(Y,s)\in K$. Then $|X-Y|\leq \diam S=\sqrt{n}\epsilon r$ and $\epsilon^2r^2\leq t-s \leq 3\epsilon^2r^2.$
By Lemma \ref{fundprop}, $$W(X-Y,t-s)\geq \phi(\sqrt{n}\epsilon r,t-s)\geq \phi(\sqrt{n}\epsilon r, 3\epsilon^2r^2)=\phi(\sqrt{n},3)(\epsilon r)^{-n},$$ where for the second inequality, we used that $\phi(\sqrt{n}\epsilon r,\cdot)$ is decreasing on $[\epsilon^2r^2/2,\infty)$. Thus, $u(X,t) = \int_K W(X-Y,t-s)\,d\mu(Y,s)\geq \phi(\sqrt{n},3)(\epsilon r)^{-n}\mu(K)$. Finally, evaluating $\phi(\sqrt{n},3)=(12\pi)^{-n/2}e^{-n/12}$, we arrive at (E2).
\end{proof}

\begin{proof}[Proof of {\eqref{estimate3}}] By assumption $\gap(Q_*,\partial_n Q)\geq \delta r$ and $Q_*$ has side length $\epsilon r$. In the argument that follows, we will impose several constraints on $\delta$ and on $\epsilon$.

Let $(X,t)\in \partial_n Q$ and $(Y,s)\in K$. Since $K\subset Q_*$, we have $\dist((X,t),(Y,s))\geq \delta r$. Dispensing with an easy case, note that if $t=s$ or $t\leq -3\epsilon^2r^2$, then $W(X-Y,t-s)=0$ and (E3) holds trivially. Thus, we may focus on the case that $t\neq s$ and $-3\epsilon^2r^2 < t\leq 0$. Then $0<|t-s|\leq 3\epsilon^2r^2<\delta^2r^2$ as long as we require $\sqrt{3}\epsilon<\delta$. Since $(X,t)$ and $(Y,s)$ are far apart, but $t$ and $s$ are close, it follows that $|X-Y|=\dist((X,t),(Y,s))\geq \delta r$. According to Lemma \ref{fundprop}, $\phi(\delta r, \cdot)$ is increasing on $(0,(1/2n)\delta^2r^2]$. Thus, by requiring $3\epsilon^2\leq (1/2n)\delta^2$, we may estimate $$W(X-Y,t-s)\leq \phi(\delta r,t-s) \leq \phi(\delta r, 3\epsilon^2r^2)= \phi(\delta/\epsilon,3)(\epsilon r)^{-n}\leq \phi(\sqrt{6n},3)(\epsilon r)^{-n}.$$ Thus, $u(X,t) \leq \phi(\sqrt{6n},3)(\epsilon r)^{-n}\mu(K)=(12\pi)^{-n/2}e^{-n/2}\mu(K)$. This yields (E3) provided that $\epsilon/\delta \leq 1/\sqrt{6n}$.
\end{proof}

Motivated by Theorem \ref{dim-lemma}, we now use Lemma \ref{alternative} to prove the following lemma at the heart of Theorem \ref{main}(ii). The proof of the corresponding statement for harmonic measure in \cite{Bourgain} (see Lemma 2) incorporated both Hausdorff and net contents, but in hindsight we found it easier to work exclusively with net contents.

\begin{lemma}\label{lemma2} For each $n\geq 1$, there exists an integer $m\geq 7$ and numbers $\rho>0$ and $\lambda>0$ with the following property. Let $\Delta=\Delta^m(\RR^{n+1})$, let $E\subset \RR^{n+1}$ be a closed set, let $(X,t)\in \Omega\equiv \RR^{n+1}\setminus E$, and let $\omega=\omega^{X,t}_{\Omega}$. If nearly every point in $\partial_e\Omega$ is regular for the Dirichlet problem for the heat equation (see \S\ref{sec:heat}), then for every cube $P\in \Delta$ such that $(X,t)\in\RR^{n+1}\setminus\overline{P}$, we have $$\Net^{n+2-\rho}_{m^{-1}\side P}(E\cap P)<(\side P)^{n+2-\rho}\quad\text{or}$$ $$\sum_{Q\in\Child(P)} \omega(Q)^{1/2} (\vol Q)^{1/2} \leq m^{-\lambda}\, \omega(P)^{1/2}(\vol P)^{1/2}.$$
\end{lemma}
\begin{proof} Let $n\geq 1$ be given, and fix $m\geq 7$ odd, $0<\rho<1$, and $0<\lambda<1$ to be specified later, ultimately depending only on $n$. Let $E$, $(X,t)$, and $\omega$ be given as in the statement. For any $j\geq 1$ and $Q\in\Delta$, let $\Child^j(Q)$ denote the set of all $j$-th generation descendents of $Q$ in $\Delta$; for example, $\Child^3(Q)=\{R\in\Delta:R\subset Q,\, \side R = m^{-3}\side Q\}$ is the set of all great grandchildren of $Q$. For the remainder of the proof, fix an integer $d\geq 1$ to be specified later. (We take $d=n+3$.)

For each $Q\in \Delta$, let $F_Q\in\Child(Q)$ denote the unique child that is at the center of $Q$ in the space-coordinates and forward-most-in-time, i.e.~it lies along $\partial_s Q$. Let $Q_*$ denote the unique child that is two layers backwards-in-time from $F_Q$, so that $(Q,F_Q,Q_*)$ form a triple of cubes to which we can apply Lemma \ref{alternative}. If $r=\side Q$ and $\epsilon=m^{-1}$, then $\epsilon r=\side F_Q=\side Q_*$ and $\gap(Q_*,\partial_n Q)=\frac{1}{2}(m-1)\epsilon r=:\delta r$. Hence $\epsilon/\delta \leq 1/\sqrt{6n}$ as soon as $m \geq 1+2\sqrt{6n}$; e.g.~$m\geq 7$ when $n=1$, $m\geq 9$ when $n=2$, and $m\geq 11$ when $n=3$. Other restrictions on $m$ will appear later in the proof (see the paragraph labeled \emph{Conclusion}, appearing below \emph{Alternative 2}). Alternatively, we may specify a congruent triple $(Q,F_Q,Q_*)$ by specifying either $F_Q\in \Delta$ or $Q_*\in \Delta$, in which case $Q$ is a ``translated'' $m$-adic cube, i.e.~$Q=Q'+R'$ for some $Q',R'\in\Delta$ with $\side R'=m^{-1}\side Q'=\side Q_*$.

Fix $0<\eta<1$ to be chosen later. By Lemma \ref{alternative}, for every triple $(Q,F_Q,Q_*)$, either one or both of the following conditions hold: \begin{equation}
\label{alt1} \omega^{Z,\tau}_{(\interior{Q})\setminus E}(E\cap \interior{Q}) \geq \eta\quad\text{for all }(Z,\tau)\in F_Q\setminus E,\end{equation} \begin{equation}\label{alt2}\Net^{n+2-\rho}_\infty(E\cap Q_*)\leq  \alpha_n m^{n+2-\rho} \eta\, (\side Q_*)^{n+2-\rho}= \alpha_n \eta\,(\side Q)^{n+2-\rho},\end{equation} where $\alpha_n>1$ depends only on $n$, since $n+2-\rho\geq n+1$ (see Remark \ref{n-only}). Note that if \eqref{alt2} holds for some triple $(Q,F_Q,Q_*)$ and $\alpha_n \eta\leq 1$, then $\Net^{n+2-\rho}_\infty(E\cap Q_*)=\Net^{n+2-\rho}_{\side Q}(E\cap Q_*)$. We make this stipulation on $\eta$.

Let $P\in\Delta$ be any cube such that $(X,t)\in\RR^{n+1}\setminus\overline{P}$. Once again, there are two alternatives.

\smallskip

\emph{Alternative 1. Suppose that \eqref{alt2} holds for some $(Q,F_Q,Q_*)$ with $Q_*\in\Child^{d+1}(P)$.} We will prove that $$\Net^{n+2-\rho}_{m^{-1}\side P}(E\cap P)<(\side P)^{n+2-\rho}.$$ Let $Q_*^{\uparrow j}\in\Delta$ denote the $j$-th ancestor of $Q_*$ in $\Delta$. (In general, $Q_*^{\uparrow 1}$ is different than $Q$, because $Q$ may be a ``translated'' cube.) Covering the set $E\cap P$ by $\Child(P)\setminus\{Q_*^{\uparrow d}\}$, $\Child(Q_*^{\uparrow d})\setminus\{Q_*^{\uparrow d-1}\}$, \dots, $\Child(Q_*^{\uparrow 1})\setminus \{Q_*\}$, and $E\cap Q_*$, we obtain \begin{equation*}\begin{split}\Net^{n+2-\rho}_{m^{-1}\side P}(E\cap P)\leq (m^{n+2}-1)(\side Q_*^{\uparrow d})^{n+2-\rho}+\cdots&+(m^{n+2}-1)(\side Q_*)^{n+2-\rho}\\
&+\alpha_n \eta\,(\side Q)^{n+2-\rho}.\end{split}\end{equation*} Rewriting each side length in terms of $\side P$ and rearranging, we have $$\frac{\Net^{n+2-\rho}_{m^{-1}\side P}(E\cap P)}{(\side P)^{n+2-\rho}} \leq (m^{n+2}-1)(m^{-(n+2-\rho)}+\cdots+m^{-(d+1)(n+2-\rho)})+\alpha_n\eta\, m^{-d(n+2-\rho)}.$$ Our goal is to choose $\eta$ and $\rho$ so that the right hand side of the previous displayed equation is strictly less than $1$. Assign $\eta=(1/2\alpha_n)m^{-(n+2)}$ so that the goal becomes \begin{equation}\label{rho-stip}(m^{n+2}-1)(m^{-(n+2-\rho)}+\cdots+m^{-(d+1)(n+2-\rho)}) +\frac{1}{2}m^{-(n+2)}m^{-d(n+2-\rho)}<1.\end{equation} Note that \eqref{rho-stip} holds at $\rho=0$, since $$(m^{n+2}-1)(m^{-(n+2)}+\cdots+m^{-(d+1)(n+2)})+\frac{1}{2}m^{-(d+1)(n+2)} = 1-\frac{1}{2} m^{-(d+1)(n+2)} < 1.$$ Thus, by continuity, \eqref{rho-stip} holds for any choice of $\rho>0$ that is sufficiently close to $0$ depending only on $m$, $n$, and $d$.

To summarize, for $0<\eta<1$ depending only on $m$ and $n$ and $0<\rho<1$ sufficiently close to $0$ depending only on $m$, $n$, and $d$, we have $\Net^{n+2-\rho}_{m^{-1}\side P}(E\cap P)<(\side P)^{n+2-\rho}.$ Below we shall see how to choose $m$ and $d$ depending only on $n$. This completes the argument for Alternative 1.

\smallskip

\emph{Alternative 2. Suppose that \eqref{alt1} holds for every $(Q,F_Q,Q_*)$ with $Q_*\in\Child^{d+1}(P)$.} The idea of the argument that follows, due to Bourgain \cite{Bourgain}, is quite simple---as many good ideas are. Here is one possible interpretation. Requiring \eqref{alt1} for every $(d+1)$-descendent $Q_*$ of $P$ is like placing a dense minefield across $P$. Hitting a mine means exiting the domain. If $m$ and $d$ are large, then a Brownian traveler sent into the past that meets $P$ is more likely to first hit one in a smaller proportion of mines near the boundary of $P$ than they are to first hit one in the larger proportion of mines near the center of $P$. Using the strong Markov property, this will allow us to prove $\sum_{Q\in\Child(P)}\omega(Q)^{1/2}(\vol Q)^{1/2} \leq m^{-\lambda} \omega(P)^{1/2}(\vol P)^{1/2}$ for some $m=m(n)$ and $\lambda=\lambda(n)$.

To proceed, we partition $P$ into annular rings of $d$-th generation descendents. Working from outside to inside, define $A_0=\emptyset$, $P_0= P$,
\begin{align*}
A_1 &= \bigcup\{Q\in\Child^d(P):Q\not\subset A_0,\, \overline{Q}\cap \partial P_0\neq\emptyset\}, & P_1&=P_0\setminus A_1,\\
A_2 &= \bigcup\{Q\in\Child^d(P):Q\not\subset A_1,\, \overline{Q}\cap \partial P_1\neq\emptyset\}, & P_2&=P_1\setminus A_2,\\
&\ \,\vdots &&\ \,\vdots\\
A_{M} &= \bigcup\{Q\in\Child^d(P):Q\not\subset A_{M-1},\, \overline{Q}\cap \partial P_{M-1}\neq\emptyset\}, & P_{M}&=P_{M-1}\setminus A_M,\\
A_{M+1}&= \bigcup\{Q\in\Child^d(P):Q\not\subset A_{M},\, \overline{Q}\cap \partial P_M \neq \emptyset\}=P_M, && \!\!\!\!\!\!\!\!\!\!P_{M+1}=\emptyset,
\end{align*} where $m^d=1+2M$. Next, for each annulus $A_i$, with $1\leq i\leq M$, choose a surface $G_i\subset A_i$ with flat faces such that (i) $G_i$ separates $\partial P_{i-1}$ from $\partial P_i$ and (ii) for any $(Z,\tau)\in G_i$, there exists an admissible triple $(Q,F_Q,Q_*)$ with $F_Q\in\Child^{d+1}(P)$, $(Z,\tau)\in F_Q$, $Q\subset A_i$, and $Q\cap A_{i+1}=\emptyset$. There are a continuum of possibilities for each $G_i$. See Figure \ref{fig:paw}.

We remark that each union $\widetilde A_j=A_{1+(j-1)m^{d-1}}\cup\dots\cup A_{m^{d-1}+(j-1)m^{d-1}}$ of $m^{d-1}$ consecutive rings $A_i$ of $d$-descendents is an annulus formed from children of $P$. In other words, $\widetilde A_1$ is the outermost annulus of children, $\widetilde A_2$ is the second annulus of children, ... and $\widetilde A_{1+(m-1)/2}$ is the innermost (degenerate) annulus of children.

\begin{figure}\begin{center}\includegraphics[width=.9\textwidth]{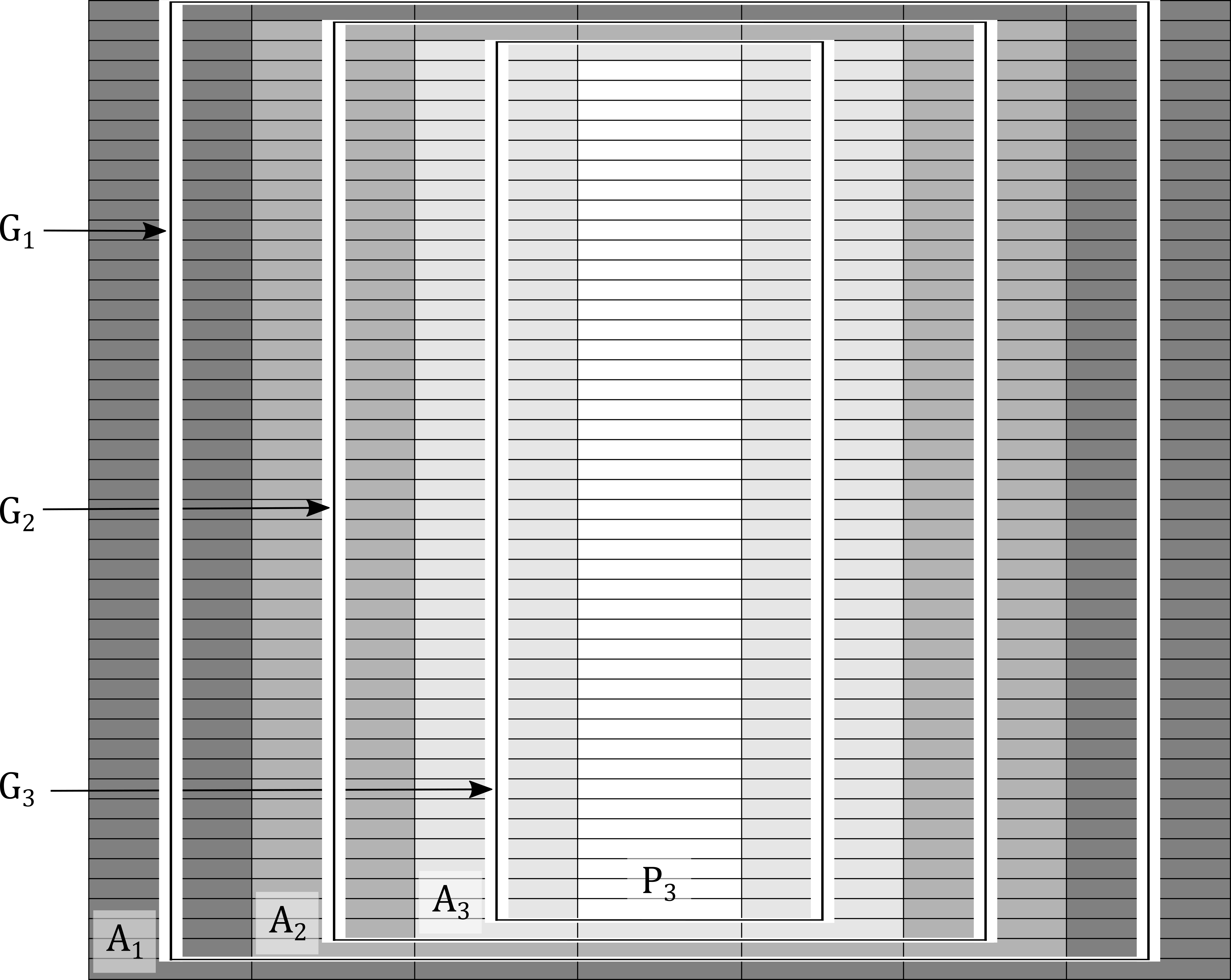}\end{center}\caption{A decomposition $P=A_1\cup A_2\cup A_3 \cup P_3$ in Alternative 2 when $n=1$, $m=7$, and $d=1$, where each wide rectangle represents a child of $P$. Brownian motion, started outside of $P$ and sent into the past, cannot meet $P_3$ without passing through surfaces $G_1$, $G_2$, and $G_3$. Each white region is the union of all $F_Q\in\Child^{2}(P)$ such that $G_i\cap F_Q\neq \emptyset$. Increasing $d$ yields a higher density $m^{d-1}$ of surfaces $G_i$ in each annulus $\widetilde A_j$ of children of $P$.}  \label{fig:paw}
\end{figure}

Next, we wish to estimate the caloric measure of certain sets in $P$. Fix any $1\leq k\leq M$. Later we will choose $k=k(n,m)$. For each $1\leq i\leq k$, let $H_i$ denote the closure of the connected component of $P\setminus G_i$ that contains $P_k$. Recall that $\Omega=\RR^{n+1}\setminus E$ and assign $G_i'=\Omega\cap\partial_e(\Omega\setminus H_i)$.

\emph{Claim 1. Caloric measure on $P_k$ is small: $\omega(P_k) \leq (1-\eta)^k\omega^{X,t}_{\Omega\setminus H_1}(G_1')$.}

To prove the claim, let $G_1,\dots, G_k$ and $G_1',\dots,G_k'$ and $H_1,\dots, H_k$ be given as above. In addition, by a slight abuse of notation, write $G_{k+1}=\partial P_k$, $H_{k+1}=\overline{P_k}$ and $G_{k+1}'=\Omega\cap\partial_e(\Omega\setminus \overline{P_k})$. Then $H_{i+1}\subset \interior{H_i}$ for all $1\leq i\leq k$. By assumption, nearly every point of $\partial_e\Omega$ is regular for the Dirichlet problem for the heat equation and $(X,t)\not\in\overline{P}$. Thus, by Corollary \ref{c-nested}, the trivial observation $\omega^{X_i,t_i}_{\Omega\setminus H_i}(G_{i+1}')\leq \omega^{X_i,t_i}_{\Omega\setminus H_i}(G_{i+1})$, and the fact that we are in Alternative 2 yield \begin{equation*}\begin{split} \label{big-nest} \omega(P_k)=\omega^{X,t}_\Omega(P_k)&\leq \omega^{X,t}_{\Omega\setminus H_1}(G_1')\prod_{i=1}^k\sup_{(X_{i},t_{i})\in G_{i}'}\omega^{X_{i},t_{i}}_{\Omega\setminus H_{i+1}}(G_{i+1})
\\ &\leq \omega^{X,t}_{\Omega\setminus H_1}(G_1')\prod_{i=1}^k\left(1-\inf_{(X_i,t_i)\in G_i'} \omega^{X_i,t_i}_{\Omega\setminus H_{i+1}}(A_i)\right) \leq \omega^{X,t}_{\Omega\setminus H_1}(G_1')(1-\eta)^k.\end{split}\end{equation*} To verify the final inequality, fix $(Z,\tau)\in G_i'$ and let $(Q,F_Q,Q_*)$ be the associated triple given by property (ii) in the definition of $G_i$. Then $\omega^{Z,\tau}_{\Omega\setminus H_{i+1}}(A_i)\geq \omega^{Z,\tau}_{\interior{Q}\setminus E}(E\cap\interior{Q}) \geq \eta$ by the maximum principle and \eqref{alt1}. QED

\emph{Claim 2. Caloric measure on $P$ is large (i.e.~not small): $\omega(P) \geq \eta\, \omega^{X,t}_{\Omega\setminus H_1}(G_1')$.}

To bound $\omega(P)$ from below, we apply the strong Markov property (Theorem \ref{t:strong-markov}) with the nested domains $\Omega\setminus H_1\subset\Omega$ and once again use the fact that we are in Alternative 2: \begin{align*}\omega(P)=\omega^{X,t}_\Omega(P) &= \omega^{X,t}_{\Omega\setminus H_1}(\partial_e\Omega\cap P) + \int_{G_1'}\omega^{Z,\tau}_\Omega(P)\,d\omega^{X,t}_{\Omega\setminus H_1}(Z,\tau)\\
&\geq \omega^{X,t}_{\Omega\setminus H_1}(G_1')\inf_{(Z,\tau)\in G_1'}\omega^{Z,\tau}_\Omega(P)\geq \eta\,\omega^{X,t}_{\Omega\setminus H_1}(G_1'), \end{align*} where the final inequality follows from the maximum principle and \eqref{alt1} as before. This concludes the proof of Claim 2.

Combining the two claims, we obtain $\omega(P_k)\leq \eta^{-1}(1-\eta)^k \omega(P)$. The upshot is that by taking $k$ (hence $m$) to be sufficiently large, we can arrange for $\omega(P_k)$ to be arbitrarily small relative to $\omega(P)$.

\emph{Conclusion (choosing the size of the grid).} Assign $d=n+3$ and $k=\lceil m^{n+2}(\log m)^2\rceil+c$, where we let $\log$ denote the natural logarithm and let $c\geq 0$ be the least integer such that $k$ is a multiple of $m^{d-1}$. Note that $k\leq M=(m^{n+3}-1)/2$ as soon as $m$ is even modestly large ($m\geq 21$ suffices). Partition the set of all children of $P$ into two collections: $$\mathcal{A}=\{Q\in\Child(P):Q\subset A_1\cup\dots\cup A_k\}\quad\text{and}\quad\mathcal{B}=\{Q\in\Child(P):Q\subset P_k\}.$$ Writing $\#\mathcal{A}=\delta m^{n+2}$, we have  \begin{align*} \sum_{Q\in\Child(P)} \omega(Q)^{1/2}(\vol Q)^{1/2} &= \sum_{Q\in\mathcal{A}} \omega(Q)^{1/2}(\vol Q)^{1/2}
+\sum_{Q\in \mathcal{B}} \omega(Q)^{1/2}(\vol Q)^{1/2}\\
&\leq \omega(A_1\cup\dots\cup A_k)^{1/2}\left(\textstyle\sum_{Q\in \mathcal{A}}\vol Q\right)^{1/2}
+ \omega(P_k)^{1/2}(\vol P_k)^{1/2}\\
&\leq \left(\delta^{1/2}+\eta^{-1/2}(1-\eta)^{k/2}\right)\omega(P)^{1/2}(\vol P)^{1/2},\end{align*} where the first inequality holds by Cauchy-Schwarz. If we can show that \begin{equation}\label{less-than-1} \delta^{1/2}+\eta^{-1/2}(1-\eta)^{k/2}<1\end{equation} for some $m$ sufficiently large depending only on $n$, then the theorem follows by imposing this requirement on $m$ and taking $\lambda = -\log_m(\delta^{1/2}+\eta^{-1/2}(1-\eta)^{k/2})$.

To finish the proof of Lemma \ref{lemma2}, let us indicate why \eqref{less-than-1} is possible. Recall that in Alternative 1 we assigned $\eta = m^{-(n+2)}/(2\alpha_n)$. Hence $\eta^{-1/2}(1-\eta)^{k/2}<1/2$ if \begin{equation}\label{less-than-2} 4\left(1-\frac{m^{-(n+2)}}{2\alpha_n}\right)^{\lceil m^{n+2}(\log m)^2\rceil+c} =4(1-\eta)^k< \eta= \frac{m^{-(n+2)}}{2\alpha_n}.\end{equation} Multiplying by $2\alpha_n$, taking logarithms, using $\log(1-x)<-x$ for $0<x<1$, and dividing through by $\log m$, we find that for \eqref{less-than-2} to hold it suffices if \begin{equation}\label{less-than-3} \frac{\log(8\alpha_n)}{\log m} - \frac{m^{-(n+2)}}{2\alpha_n} \left( m^{n+2}\log{m}\right) < -(n+2).\end{equation} This is plainly true if $m$ is sufficiently large, depending only on $n$. (This is the rationale for our choice of $k$ and $d$. We could not have verified \eqref{less-than-3} if we had chosen $k\lesssim m^{n+2}\log m$.)
Finally, looking at Figure \ref{fig:paw} as a guide, the reader may verify that $$\frac{\#B}{m^{n+2}}=\frac{\vol P_k}{\vol P} = \frac{(m^d-2k)^n(m^{2d}-2k)}{m^{d(n+2)}}\rightarrow 1\quad\text{as $m \rightarrow\infty$},$$ since $k=o(m^d)$. Thus, $$\delta=\frac{\#\mathcal{A}}{m^{n+2}} =\frac{m^{n+2}-\#\mathcal{B}}{m^{n+2}} = 1-\frac{\#\mathcal{B}}{m^{n+2}}<1/2$$ if $m$ is sufficiently large depending only on $n$. \end{proof}

We can now complete the proof of the main theorem.

\begin{proof}[Proof of {Theorem \ref{main}(ii)}] Let $\Omega\subsetneq\RR^{n+1}$ be a domain, let $E=\RR^{n+1}\setminus \Omega$, let $(X,t)\in \Omega$, and let $\omega=\omega^{X,t}_\Omega$. Assume that nearly every point in $\partial_e\Omega$ is regular for the Dirichlet problem for the heat equation. Let $m=m(n)\geq 7$, $\rho=\rho(n)>0$, and $\lambda=\lambda(n)>0$ be constants from Lemma \ref{lemma2} and set $\beta_n=\lambda\rho/(\lambda+\rho)$. We will prove that $$\overline{\dim}_H\,\omega\leq n+2-\beta_n.$$ Since caloric measure is carried by the past (Remark \ref{zero-future}) and Hausdorff dimension is countably stable (Remark \ref{stable}), it suffices to prove that $$\overline{\dim}_H\,\omega\res\{(Y,s):s<t-1/k\}\leq n+2-\beta_n\quad\text{for all $k\geq 1$.}$$ In fact, since Hausdorff dimension is invariant under translation and (parabolic) dilation, we may assume that $(X,t)=(0,\dots,0,1)$ and prove that $\mu=\omega\res\{(Y,s):s<0\}$ has upper Hausdorff dimension at most $n+2-\beta_n$. Fix $\Delta=\Delta^m(\RR^{n+1})$ with $m$ as above and let $P\in\Delta$ with $\side P\leq 1$. On one hand, if $(X,t)\not\in\overline{P}$, then Lemma \ref{lemma2} ensures that $\mathcal{M}^{n+2-\rho}(E\cap P)<(\side P)^{n+2-\rho}$ or $\sum_{Q\in\Child(P)}\mu(Q)^{1/2}(\vol Q)^{1/2} \leq m^{-\lambda}\mu(P)^{1/2}(\vol P)^{1/2}$. On the other hand, if $(X,t)\in\overline{P}$, then $P\subset \RR^n\times[0,\infty)$. Hence $\mu(P)=0$ and $\sum_{Q\in\Child(P)}\mu(Q)^{1/2}(\vol Q)^{1/2} \leq m^{-\lambda} \mu(P)^{1/2}(\vol P)^{1/2}$ holds trivially. We have checked the hypothesis of Theorem \ref{dim-lemma} for all $P\in\Delta$ with $\side P\leq 1$. Therefore, $\overline{\dim}_H\,\mu\leq n+2-\beta_n$ by Theorem \ref{dim-lemma}.\end{proof}

\section{Bourgain constants for harmonic and caloric measures} \label{sec:constants}

For all $n\geq 2$, let $b_n\in[0,1]$ denote Bourgain's constant for harmonic measure, i.e.~the largest number such that the upper Hausdorff dimension of harmonic measure is at most $n-b_n$ for all domains $\Omega\subset\RR^n$. As we noted in the introduction, $b_2=1$ and $0<b_n<1$ for all $n\geq 3$ by theorems of Jones and Wolff \cite{Jones-Wolff}, Bourgain \cite{Bourgain}, and Wolff \cite{Wolff}. The definition also makes sense when $n=1$; since harmonic measure on an open interval $I\subsetneq\RR$ is a weighted sum of Dirac masses (as the only harmonic functions on $I\subset\RR$ are the linear functions $u(x)=mx+b$), $\overline{\dim}_H\, \omega^X_I = 0$ for all $I$, whence $b_1=1$.

Let $\beta_n$ denote the optimal constant in Theorem \ref{main}(ii). Since steady state solutions of the heat equation are harmonic, it is natural to ask whether there is a relationship between $b_n$ and $\beta_n$? The answer is yes. To see this, we first need to examine the parabolic Hausdorff dimension of caloric measure on cylindrical domains.

\begin{lemma}[caloric measure on cylindrical domains]\label{cylinder-dim} Let $n\geq 1$, let $D\subset\RR^n$ be a bounded Euclidean domain, and let $\Omega=D\times\RR\subset\RR^{n+1}$. Then: \begin{enumerate}
\item $\overline{\dim}_H\, \omega^{X,t}_\Omega = \overline{\dim}_H\, \omega^{Y,s}_\Omega$ for all $(X,t),(Y,s)\in\Omega$;
\item $\omega^{X,t}_\Omega(A\times \RR)=\omega^X_D(A)$ for all Borel sets $A\subset \partial D$;
\item $(X_0,t_0)\in\partial \Omega$ is regular for the Dirichlet problem for the heat equation on $\Omega$ iff $X_0\in\partial D$ is regular for the Dirichlet problem for harmonic functions on $D$;
\item if $\omega_D=\omega^{X}_{D}$ is harmonic measure on $D$ and $\omega=\omega^{X,t}_{\Omega}$ is caloric measure on $\Omega$, then $\overline{\dim}_H\, \omega = \overline{\dim_H}\, \omega_D + 2$.
\end{enumerate}
\end{lemma}

\begin{proof} To prove (i), fix two poles $(X,t),(Y,s)\in\Omega$ and suppose that $\omega^{X,t}_\Omega(\partial\Omega\setminus E)=0$ for some Borel set $E$. We will find a Borel set $F$ such that $\dim_H F = \dim_H E$ and $\omega^{Y,s}_\Omega(\partial\Omega\setminus F)=0$. Choose any $r\geq 0$ such that $t+r>s$ and assign $F=E+re_{n+1}$. On one hand, $\omega^{X,t+r}_\Omega(\partial\Omega\setminus F)=\omega^{X,t}_\Omega(\partial\Omega\setminus E)=0$ by invariance of solutions to the heat equation on $\Omega$ under translations in time. On the other hand, $\omega^{Y,s}_\Omega\ll\omega^{X,t+r}_\Omega$ by Remark \ref{r:harnack}. Hence $\omega^{Y,s}_\Omega(\partial\Omega\setminus F)=0$. Since $\dim_H F=\dim_H E$, we conclude that $\overline{\dim_H}\, \omega^{Y,s}_\Omega\leq \overline{\dim_H}\, \omega^{X,t}_\Omega$. Swapping the roles of $(X,t)$ and $(Y,s)$ yields (i).

For items (ii) and (iii), see \cite[p.~335]{Doob}. In particular, if $A\subset \partial D$ is any Borel set with full harmonic measure on $D$, then $A\times\RR$ has full caloric measure on $\Omega$. Hence $\overline{\dim}_H\,\omega \leq \dim_H(A\times \RR)=\dim_H(A)+\dim_H(\RR)=\dim_H(A)+2$ by Lemma \ref{l:products}. (Recall that we are computing Hausdorff dimensions relative to the parabolic metric on $\RR^{n+1}$.) Letting $A$ vary, we obtain $\overline{\dim}_H\, \omega \leq \overline{\dim}_H\,\omega_D+2.$ This gives half of (iv).

For the reverse inequality, suppose that $E\subset\partial\Omega$ is a Borel set and consider the set of time-slices $E_\tau=E\cap (\RR^n\times\{\tau\})$. If there is $\epsilon>0$ and $T\subset(-\infty,t)$ with $\Haus^2(T)>0$ such that $\omega_D(E_\tau)<1-\epsilon$ for each $\tau\in T$, then $\omega(\bigcup_{\tau\in T} E_\tau)<(1-\epsilon)\,\omega(\partial D\times T).$ This can be shown using the transition density for Brownian motion started at $X$, stopped when it exits $D$ (see e.g.~\cite[p.~593]{Doob}). Thus, if $E$ has full $\omega$ measure, then $\omega_D(E_\tau)=1$, $\dim_H E_\tau \geq \overline{\dim}_H\,\omega_D$ for $\Haus^2$-a.e.~$\tau\in(-\infty,t)$, and $\dim_H E\geq \overline{\dim}_H\,\omega_D + 2$. Therefore, $\overline{\dim}_H\, \omega \geq \overline{\dim}_H\,\omega_D + 2$. This completes the proof of (iv).
\end{proof}

\begin{lemma} For all $n\geq 1$, we have $\beta_n\leq b_n$.\end{lemma}

\begin{proof} The value of Bourgain's constant is unchanged if we demand that the domains in question are bounded. Given $\epsilon>0$, choose a bounded domain $D\subset\RR^n$ with harmonic measure $\omega_D=\omega^{X}_D$ such that $\overline{\dim}_H\,\omega_D\geq n-b_n-\epsilon$. It follows from Lemma \ref{cylinder-dim}(iii) that nearly every point of $\partial_e\Omega=\partial D\times\RR$ of the cylinder domain $\Omega=D\times\RR$ is regular for the Dirichlet problem for the heat equation on $\Omega$. By Lemma \ref{cylinder-dim}(iv) and the definition of $\beta_n$, we have $n+2-\beta_n\geq \overline{\dim}_H\, \omega^{X,0}_\Omega = \overline{\dim}_H\,\omega_D+2 > n+2-b_n-\epsilon.$ Sending $\epsilon\rightarrow 0$ and rearranging, we discover $\beta_n\leq b_n$.\end{proof}

\begin{corollary} We have $\beta_1\leq 1$, $\beta_2\leq 1$, and $\beta_n\leq b_n$ for all $n\geq 3$.\end{corollary}

Although the value of $b_n$ is currently unknown when $n\geq 3$, we may still ask:

\begin{question} Is $\beta_n=b_n$?\end{question}

\renewcommand{\baselinestretch}{1.05}

\providecommand{\bysame}{\leavevmode\hbox to3em{\hrulefill}\thinspace}
\providecommand{\MR}{\relax\ifhmode\unskip\space\fi MR }
\providecommand{\MRhref}[2]{%
  \href{http://www.ams.org/mathscinet-getitem?mr=#1}{#2}
}
\providecommand{\href}[2]{#2}


\begin{thebibliography}{GLFS05}

\bibitem[AAM19]{AAM19}
M.~Akman, J.~Azzam, and M.~Mourgoglou, \emph{Absolute continuity of
  harmonic measure for domains with lower regular boundaries}, Adv. Math.
  \textbf{345} (2019), 1206--1252. \MR{3903916}

\bibitem[ALV15]{ALV15}
M.~Akman, J.~Lewis, and A.~Vogel, \emph{Hausdorff dimension and
  {$\sigma$} finiteness of {$p$} harmonic measures in space when {$p\geq n$}},
  Nonlinear Anal. \textbf{129} (2015), 198--216. \MR{3414928}

\bibitem[AMT17]{amt-example}
J.~Azzam, M.~Mourgoglou, and X.~Tolsa, \emph{Singular sets for
  harmonic measure on locally flat domains with locally finite surface
  measure}, Int. Math. Res. Not. IMRN (2017), no.~12, 3751--3773. \MR{3693664}

\bibitem[Azz18]{azzam-corkscrews}
J.~Azzam, \emph{Tangents, rectifiability, and corkscrew domains}, Publ. Mat.
  \textbf{62} (2018), no.~1, 161--176. \MR{3738187}

\bibitem[Azz20]{azzam-dimension-drop}
J.~Azzam, \emph{Dimension drop for harmonic measure on {A}hlfors regular
  boundaries}, Potential Anal. \textbf{53} (2020), no.~3, 1025--1041.
  \MR{4140087}

\bibitem[Bad11]{Badger-thesis}
M.~Badger, \emph{Harmonic {P}olynomials and {F}ree {B}oundary {R}egularity
  for {H}armonic {M}easure from {T}wo {S}ides}, ProQuest LLC, Ann Arbor, MI,
  2011, Thesis (Ph.D.)--University of Washington. \MR{2982196}

\bibitem[Bad12]{Badger-nullsets}
M.~Badger, \emph{Null sets of harmonic measure on {NTA} domains:
  {L}ipschitz approximation revisited}, Math. Z. \textbf{270} (2012), no.~1-2,
  241--262. \MR{2875832}

\bibitem[Bat96]{cantor-dim}
A.~Batakis, \emph{Harmonic measure of some {C}antor type sets}, Ann.
  Acad. Sci. Fenn. Math. \textbf{21} (1996), no.~2, 255--270. \MR{1404086}

\bibitem[BG22]{BG-decimals}
M.~Badger and A.~Genschaw, \emph{Lower bounds on {B}ourgain's constant
  for harmonic measure}, preprint, \textsf{arXiv:2205.15101v2}, 2022.

\bibitem[Bis92]{Bishop-questions}
C.~J.~Bishop, \emph{Some questions concerning harmonic measure},
  Partial differential equations with minimal smoothness and applications
  ({C}hicago, {IL}, 1990), IMA Vol. Math. Appl., vol.~42, Springer, New York,
  1992, pp.~89--97. \MR{1155854}

\bibitem[Bis07]{bishop-review}
C.~J.~Bishop, \emph{Review of \emph{Harmonic measure}}, Bull. Amer.
  Math. Soc. \textbf{44} (2007), no.~2, 267--276.

\bibitem[Bou87]{Bourgain}
J.~Bourgain, \emph{On the {H}ausdorff dimension of harmonic measure in higher
  dimension}, Invent. Math. \textbf{87} (1987), no.~3, 477--483. \MR{874032}

\bibitem[DEM21]{DEM-magic}
G.~David, M.~Engelstein, and S.~Mayboroda, \emph{Square functions,
  nontangential limits, and harmonic measure in codimension larger than 1},
  Duke Math. J. \textbf{170} (2021), no.~3, 455--501. \MR{4255042}

\bibitem[Doo01]{Doob}
J.~L.~Doob, \emph{Classical potential theory and its probabilistic
  counterpart}, Classics in Mathematics, Springer-Verlag, Berlin, 2001, Reprint
  of the 1984 edition. \MR{1814344}

\bibitem[EG82]{EG-Wiener}
L.~C.~Evans and R.~F.~Gariepy, \emph{Wiener's criterion for the heat
  equation}, Arch. Rational Mech. Anal. \textbf{78} (1982), no.~4, 293--314.
  \MR{653544}

\bibitem[EG92]{EG}
L.~C.~Evans and R.~F.~Gariepy, \emph{Measure theory and fine
  properties of functions}, Studies in Advanced Mathematics, CRC Press, Boca
  Raton, FL, 1992. \MR{1158660}

\bibitem[Eng17]{Engelstein-parabolic}
M.~Engelstein, \emph{A free boundary problem for the parabolic {P}oisson
  kernel}, Adv. Math. \textbf{314} (2017), 835--947. \MR{3658732}

\bibitem[GH20]{Genschaw-Hofmann}
A.~Genschaw and S.~Hofmann, \emph{A {W}eak {R}everse {H}\"{o}lder
  {I}nequality for {C}aloric {M}easure}, J. Geom. Anal. \textbf{30} (2020),
  no.~2, 1530--1564. \MR{4081323}

\bibitem[GLFS05]{grebenkov-numerical-experiment}
D.~S.~Grebenkov, A.~A.~Lebedev, M.~Filoche, and B.~Sapoval, \emph{Multifractal
  properties of the harmonic measure on koch boundaries in two and three
  dimensions}, Phys. Rev. E \textbf{71} (2005), 056121.

\bibitem[GM05]{GM}
J.~B.~Garnett and D.~E.~Marshall, \emph{Harmonic measure}, New
  Mathematical Monographs, vol.~2, Cambridge University Press, Cambridge, 2005.
  \MR{2150803}

\bibitem[HLN04]{HLN}
S.~Hofmann, J.~L.~Lewis, and K.~Nystr\"{o}m, \emph{Caloric measure in
  parabolic flat domains}, Duke Math. J. \textbf{122} (2004), no.~2, 281--346.
  \MR{2053754}

\bibitem[JW88]{Jones-Wolff}
P.~W.~Jones and T.~H.~Wolff, \emph{Hausdorff dimension of harmonic
  measures in the plane}, Acta Math. \textbf{161} (1988), no.~1-2, 131--144.
  \MR{962097}

\bibitem[KPT09]{kenigpreisstoro}
C.~Kenig, D.~Preiss, and T.~Toro, \emph{Boundary structure and size in terms of
  interior and exterior harmonic measures in higher dimensions}, J. Amer. Math.
  Soc. \textbf{22} (2009), no.~3, 771--796. \MR{2505300}

\bibitem[KW82]{Kaufman-Wu-82}
R.~Kaufman and J.-M.~Wu, \emph{Parabolic potential theory}, J.
  Differential Equations \textbf{43} (1982), no.~2, 204--234. \MR{647063}

\bibitem[Lan73]{Lanconelli-Wiener}
E.~Lanconelli, \emph{Sul problema di {D}irichlet per l'equazione del
  calore}, Ann. Mat. Pura Appl. (4) \textbf{97} (1973), 83--114. \MR{372226}

\bibitem[LM95]{Lewis-Murray}
J.~L.~Lewis and M.~A.~M.~Murray, \emph{The method of layer potentials
  for the heat equation in time-varying domains}, Mem. Amer. Math. Soc.
  \textbf{114} (1995), no.~545, viii+157. \MR{1323804}

\bibitem[LVV05]{LVV}
J.~L.~Lewis, G.~C.~Verchota, and A.~L.~Vogel, \emph{Wolff
  snowflakes}, Pacific J. Math. \textbf{218} (2005), no.~1, 139--166.
  \MR{2224593}

\bibitem[Mak85]{Makarov}
N.~G.~Makarov, \emph{On the distortion of boundary sets under conformal
  mappings}, Proc. London Math. Soc. (3) \textbf{51} (1985), no.~2, 369--384.
  \MR{794117}

\bibitem[Mat95]{Mattila}
P.~Mattila, \emph{Geometry of sets and measures in {E}uclidean spaces},
  Cambridge Studies in Advanced Mathematics, vol.~44, Cambridge University
  Press, Cambridge, 1995, Fractals and rectifiability. \MR{1333890 (96h:28006)}

\bibitem[MMR00]{measure-dim}
P.~Mattila, M.~Mor\'{a}n, and J.-M.~Rey, \emph{Dimension of a
  measure}, Studia Math. \textbf{142} (2000), no.~3, 219--233. \MR{1792606}

\bibitem[MP21]{MP-caloric}
M.~Mourgoglou and C.~Puliatti, \emph{Blow-ups of caloric measure in
  time varying domains and applications to two-phase problems}, J. Math. Pures
  Appl. (9) \textbf{152} (2021), 1--68. \MR{4280831}

\bibitem[Nev35]{Nevanlinna}
R.~Nevanlinna, \emph{{Das harmonische Ma{\ss} von Punktmengen und seine
  Anwendung in der Funktionentheorie}}, Comptes rendus du huiti{\`e}me
  Congr{\`e}s des math{\'e}maticiens Scandinaves: tenu a Stockholm 14-18
  Ao{\^u}t 1934, H{\aa}kan Ohlssons Boktryckeri, Lund, 1935, pp.~116--133
  (German).

\bibitem[Rog98]{Rogers}
C.~A.~Rogers, \emph{Hausdorff measures}, Cambridge Mathematical Library,
  Cambridge University Press, Cambridge, 1998, Reprint of the 1970 original,
  With a foreword by K. J. Falconer. \MR{1692618}

\bibitem[Suz80]{Suzuki}
N.~Suzuki, \emph{On the essential boundary and supports of harmonic
  measures for the heat equation}, Proc. Japan Acad. Ser. A Math. Sci.
  \textbf{56} (1980), no.~8, 381--385. \MR{596009}

\bibitem[Swe92]{Sweezy-planar}
C.~Sweezy, \emph{The {H}ausdorff dimension of elliptic measure---a
  counterexample to the {O}ksendahl conjecture in {${\bf R}^2$}}, Proc. Amer.
  Math. Soc. \textbf{116} (1992), no.~2, 361--368. \MR{1161401}

\bibitem[Swe94]{Sweezy-space}
C.~Sweezy, \emph{The {H}ausdorff dimension of elliptic and
  elliptic-caloric measure in {${\bf R}^N,\;N\geq 3$}}, Proc. Amer. Math. Soc.
  \textbf{121} (1994), no.~3, 787--793. \MR{1186138}

\bibitem[Swe96]{Sweezy-ac}
C.~Sweezy, \emph{Absolute continuity for elliptic-caloric measures},
  Studia Math. \textbf{120} (1996), no.~2, 95--112. \MR{1406531}

\bibitem[TW85]{TW85}
S.~J.~Taylor and N.~A.~Watson, \emph{A {H}ausdorff measure classification of
  polar sets for the heat equation}, Math. Proc. Cambridge Philos. Soc.
  \textbf{97} (1985), no.~2, 325--344. \MR{771826}

\bibitem[Wat12a]{Watson-open}
N.~A.~Watson, \emph{Caloric measure for arbitrary open sets}, J. Aust. Math.
  Soc. \textbf{92} (2012), no.~3, 391--407. \MR{3044478}

\bibitem[Wat12b]{Watson}
N.~A.~Watson, \emph{Introduction to heat potential theory}, Mathematical
  Surveys and Monographs, vol. 182, American Mathematical Society, Providence,
  RI, 2012. \MR{2907452}

\bibitem[Wol93]{Wolff-sigma}
T.~H.~Wolff, \emph{Plane harmonic measures live on sets of
  {$\sigma$}-finite length}, Ark. Mat. \textbf{31} (1993), no.~1, 137--172.
  \MR{1230270}

\bibitem[Wol95]{Wolff}
T.~H.~Wolff, \emph{Counterexamples with harmonic gradients in {${\bf
  R}^3$}}, Essays on {F}ourier analysis in honor of {E}lias {M}. {S}tein
  ({P}rinceton, {NJ}, 1991), Princeton Math. Ser., vol.~42, Princeton Univ.
  Press, Princeton, NJ, 1995, pp.~321--384. \MR{1315554}

\bibitem[Wu86]{wu-singularity}
J.-M.~Wu, \emph{On singularity of harmonic measure in space}, Pacific J.
  Math. \textbf{121} (1986), no.~2, 485--496. \MR{819202}

\end{thebibliography}
\end{document}